%% file: NCstacks.tex
\documentclass[11pt,twoside,a4paper]{amsart}

\usepackage{amssymb}
\usepackage{vmargin}
\usepackage{amscd}
\usepackage{stmaryrd}
\usepackage{mathrsfs}
\usepackage[all]{xy}
\usepackage{xr}%%%or {xr-hyper}
\usepackage{hyperref}
\hypersetup{colorlinks,
linkcolor=black,
citecolor=black,
urlcolor=blue}

\setmargins{32mm}{20mm}{14.6cm}{22cm}{1cm}{1cm}{1cm}{1cm}

\include{newcommands}

\include{newthms}

\begin{document}

\begin{abstract}
 We introduce a formalism for derived moduli functors on differential graded associative algebras, which leads to non-commutative enhancements of derived moduli stacks and naturally gives rise to structures such as Hall algebras. Descent arguments are not available in the non-commutative context, so we establish new methods for constructing various kinds of atlases. The formalism  permits the development of the theory of shifted bi-symplectic and shifted double Poisson structures in the companion paper \cite{NCpoisson}. 
\end{abstract}

\title{Non-commutative derived moduli prestacks}
\author{J.P.Pridham}

\maketitle
\section*{Introduction}

The basic building blocks for derived algebraic geometry in characteristic $0$ are  differential graded-commutative  algebras (CDGAs) $A_0 \xla{\delta} A_1 \xla{\delta} \ldots$ (where we use homological grading) concentrated in non-negative chain degrees. In particular, derived affine schemes form the opposite category to these, localised at weak equivalences, and derived moduli problems give rise to functors from such CDGAs to  $\infty$-groupoids, or equivalently to simplicial sets.

This paper is concerned with non-commutative enhancements of derived moduli functors, which tend to exist when the moduli problems are linear in nature. Specifically, we look at functors $F$ on   differential graded associative algebras (DGAAs) concentrated in non-negative chain degrees,
which restrict to the usual derived moduli functors $F^{\comm}$ when restricted to CDGAs, and to the non-commutative deformation functors of \cite{EfimovLuntsOrlov1} when restricted to Artinian local DGAAs. The main motivating examples are given by moduli of perfect complexes and related constructions: given any DGAA (or even dg category) $A$ , we can consider the functor which sends a  DGAA $B$ to the $\infty$-groupoid of perfect $A \otimes B$-complexes. Related constructions are given by considering the $\infty$-groupoid of Morita morphisms from $A$ to $B$, or the $\infty$-groupoid of DGAA maps from $A$ to $B$ modulo inner automorphisms.

From a non-commutative  moduli functor $F$ and a finite  DGAA $M$, we can also form a new moduli functor $\Pi_MF:=F(M\circ -)$. The obvious examples to consider for $M$ are the matrix algebras $M_n = \Mat_n$, which lead to representation spaces given by  the commutative restrictions $(\Pi_{M_n}F)^{\comm}$. In particular, this means that the usual process of studying an associative algebra by looking at 
the associated representation spaces becomes a two-stage process,  with our non-commutative moduli functors featuring as an intermediate step. Another example of this form is the functor $F((\begin{smallmatrix} \Z & \Z\\ 0 & \Z \end{smallmatrix})\ten -)$ parametrising extensions of elements of $F$, which leads to the data required to recover Hall algebras. 

A significant difference between commutative and non-commutative settings is that there is no good notion of descent in the non-commutative setting, because coproducts are not preserved by pullbacks. This might seem at first sight to be a serious stumbling block, but sheaves and descent are only of secondary importance in derived algebraic geometry, where the conditions for existence of cotangent complexes with good obstruction theory are essentially independent of sheaf conditions. 
We therefore just concentrate on  presheaves, and look for atlases consisting of a nested union of  affine objects which map  surjectively to our presheaf, without any form of descent. For moduli of projective modules or of perfect complexes, we show how to construct such atlases by parametrising idempotents. For more general well-behaved derived NC prestacks $F$, we prove a weaker result showing that every point $x \in F(A)$ of the prestack is dominated by a form of non-commutative Lie algebroid \'etale over $F$, since this is all we need  to develop a theory of shifted double Poisson structures in  \cite{NCpoisson}.

The structure of the paper is as follows.

In \S \ref{NCmodulisn}, we introduce NC prestacks and their derived analogues, together with cotangent complexes, \'etale morphisms and submersive morphisms (the NC analogue of smooth morphisms). We also explain how each such prestack has  associated representation spaces, look at how they can recover Hall algebras, and explain the relations with various other non-commutative enhancements of algebraic geometry. 

In \S \ref{NCstacksn}, we then introduce non-commutative analogues of higher and derived Deligne--Mumford and Artin stacks, which can be thought of as Kan complexes of non-commutative affine schemes, following the philosophy of \cite{stacks2}. The main results are Proposition \ref{projgeomprop} and Corollary \ref{Perfgeomcor}, and their derived analogues  Propositions \ref{dprojgeomprop} and \ref{dPerfgeomcor}, showing that the moduli functors of projective modules and of perfect complexes are derived NC Artin $\infty$-prestacks. We also observe in Remark \ref{Perfnicermk} that moduli of perfect complexes over any locally proper dg category  have well-behaved cotangent complexes.

Since constructing Artin atlases is much harder work without descent, \S \ref{stackysn} instead shows how to approximate derived NC prestacks by stacky DGAAs, which act as a derived non-commutative analogue of Lie algebroids. A stacky DGAA is a bidifferential bigraded associative algebra, where we think of one grading as stacky and the other as derived, and only define equivalences relative to the derived structure. Every derived NC prestack $F$ naturally extends to a functor $D_*F$ on stacky DGAAs, and when $F$ has a well-behaved cotangent complex, we concentrate on the rigid points  $(D_*F)_{\rig} \subset D_*F$, i.e. the points corresponding to suitably \'etale maps from stacky DGAAs. The main results are    Corollary \ref{etsitecor} and Proposition \ref{replaceprop},  showing that most invariants of $F$ can be recovered from $(D_*F)_{\rig}$; in \cite{NCpoisson}, we use these to extend the definition of shifted bisymplectic and double Poisson structures naturally to all such derived NC prestacks.

\subsubsection*{Notation}

For a chain (resp. cochain) complex $M$, we write $M_{[i]}$ (resp. $M^{[j]}$) for the complex $(M_{[i]})_m= M_{i+m}$ (resp. $(M^{[j]})^m = M^{j+m}$). We often work with double complexes, in the form of cochain chain complexes, in which case $M_{[i]}^{[j]}$ is the double complex $(M_{[i]}^{[j]})^n_m= M_{i+m}^{j+n}$. When we have a single grading and need to compare chain and cochain complexes, we silently  make use of the equivalence  $u$ from chain complexes to cochain complexes given by $(uV)^i := V_{-i}$, and refer to this as rewriting the chain complex as a cochain complex (or vice versa). On suspensions, this has the effect that $u(V_{[n]}) = (uV)^{[-n]}$. We also occasionally write $M[i]:=M^{[i]} =M_{[-i]}$ when there is only one grading. 

For chain complexes, by default we denote differentials by $\delta$. When we work with cochain chain complexes, the cochain differential is usually denoted by $\pd$. We use the subscripts and superscripts $\bt$ to emphasise that chain and cochain complexes incorporate differentials, with $\#$ used instead when we are working  with the underlying graded objects.

Given $A$-modules $M,N$ in chain complexes, we write $\HHom_A(M,N)$ for the cochain complex given by
\[
 \HHom_A(M,N)^i= \Hom_{A_{\#}}(M_{\#},N_{\#[-i]}),
\]
with differential $\delta f= \delta_N \circ f \pm f \circ \delta_M$,
where $V_{\#}$ denotes the graded vector space underlying a chain complex $V$.

We write $s\Set$ for the category of simplicial sets, and $\map$ for derived mapping spaces, i.e. the right-derived functor of simplicial $\Hom$. (For dg categories, $\map$ corresponds via the Dold--Kan correspondence to the truncation of $\oR\HHom$.)

\tableofcontents

\section{NC moduli functors}\label{NCmodulisn}

\subsection{The setup}\label{setupsn}

% %
% See {off} for well-known problems with these, primarily stability axiom. We don't have faithfully flat descnet, essentially because coproduct isn't tensor product any more. We can still apply \cite{stacks2}, though, taking $\oP$ to be submersive morphisms and $\eps$-morphisms to be epimorphisms of presheaves (so maps with a section). For non-(strongly quasi-compact), we take ind instead of disjoint unions, requiring all maps to be open (i.e. \'etlae monomorphisms), which tallies with usu comm approach. Key for this to work is that we can represent perfect complexes without needing descent, but we will have to encode idempotents. 

\begin{definition}\label{algdef}
 For a  commutative ring $R$, we write $\Alg(R)$ for the category of associative  $R$-algebras $A$.

We denote the opposite category to  $\Alg(R)$  by $\Aff^{nc}(R)$. Given $A \in \Alg(R)$, we denote the corresponding object of  $\Aff^{nc}(R)$ by $\Spec^{nc} A$.
\end{definition}

\begin{definition}\label{dgalgdef}
For a  CDGA $R= ( \ldots \xra{\delta}R_1 \xra{\delta} R_0) $, we write $dg\Alg(R)$ for the category of associative  $R$-algebras $A_{\bt}= (\ldots \xra{\delta} A_1 \xra{\delta} A_0 \xra{\delta} \ldots)  $ in  chain complexes. We write $dg_+\Alg(R)$ for the full subcategory consisting of objects $A$ concentrated in non-negative chain degrees.  

We denote the opposite category to  $dg\Alg(R)$ (resp. $dg_+\Alg(R)$) by $DG\Aff^{nc}(R)$ (resp. $DG^+\Aff^{nc}(R)$). Given $A \in dg_+\Alg(R)$, we denote the corresponding object of  $DG^+\Aff^{nc}(R)$ by $\Spec^{nc} A$.
\end{definition}

We think of $DG^+\Aff^{nc}(R)$ as the category of non-commutative derived affine schemes over $R$, with equivalences given by quasi-isomorphism.

There is a model structure on $dg_+\Alg(R)$ (resp. $dg\Alg(R)$) in which weak equivalences are quasi-isomorphisms and fibrations are surjections in strictly positive degrees (resp. surjections). The model structures come from the projective model structure on chain complexes of abelian groups by applying \cite[Theorem 11.3.2]{Hirschhorn} to the free-forgetful adjunction. The inclusion functor $dg_+\Alg(R) \to dg\Alg(R)$ is then left Quillen, with right adjoint given by good truncation. In this model structure, objects of $dg_+\Alg(R)$ which are freely generated as graded associative $R_{\#}$-algebras (forgetting the differential)  are cofibrant, as are their retracts. %%note that we aren't working over $\Q$: free functor sends $V$ to $R\ten_{\Z}T_{\Z}(V)$, so trivial cofibrations do map to quasi-isomorphisms.

\begin{remark}
 
Note that we do not need  any restrictions on the characteristic of $R$, since the free functor from chain complexes of abelian groups to $R$-algebras is given by the tensor functor $V \mapsto R\ten_{\Z}T_{\Z}(V)$, which sends trivial cofibrations to quasi-isomorphisms. %%argt: if $V=U\oplus W$ for $W$ acyclic, then $F(V)$ is direct sum of strings of $\ten$ of $U$ and $W$, and we can homotope away anything involving a $W$.
This is in marked contrast to the situation for commutative algebras (or indeed algebras over any symmetric operad), where differential graded algebras do not have a natural model structure.

\end{remark}

The following  notion gives the classical truncation of a derived NC affine scheme:
\begin{definition}\label{pi0def}
 Define $\pi^0 \co DG^+\Aff^{nc}(R) \to \Aff^{nc}(R)$ by $\pi^0\Spec^{nc}A := \Spec^{nc}\H_0A$.
\end{definition}

The reason for this notation is that $dg_+\Alg(R)$ is Quillen equivalent, via Dold--Kan normalisation and the Eilenberg--Zilber shuffle product, to the natural model category of simplicial $R$-algebras. Thus objects $\Spec^{nc}A \in DG^+\Aff^{nc}(R)$ correspond to cosimplicial diagrams in $ \Aff^{nc}(R)$, with $\pi^0\Spec^{nc}A$ being the limit of the cosimplicial diagram (whereas path components $\pi_0$ denote the colimit of a simplicial diagram).

\begin{definition}\label{Aedef}
 Given $A \in dg\Alg(R)$, we denote by $A^{\op}$ the $R$-DGAA with the same underlying complex as $A$ but the opposite multiplication, so $a^{\op}b^{\op}:= (-1)^{\bar{a}\bar{b}} (ba)^{\op}$. We then write $A^e:=A\ten_RA^{\op}$ and  $A^{\oL,e}:= A\ten_R^{\oL}A^{\op}$, which we may regard as an object of the homotopy category of $dg\Alg(R)$, so $A^{e}$-modules (resp. $A^{\oL,e}$-modules) are $A$-bimodules for which the left and right $R$-module structures are strictly  compatible (resp. compatible up to coherent homotopy).
\end{definition}
Note that if $A$ and $R$ are both concentrated in non-negative chain degrees, with the underlying graded algebra $A_{\#}$ flat over $R_{\#}$, then we may simply take $A^{\oL,e}= A^e$. In particular, this holds whenever $A$ is cofibrant in $dg\Alg(R)$. %%note that left and right $R$-modules structures on $A$ are the same, by hypothesis.

\begin{definition}
 Given $A \in dg\Alg(R)$, we denote by $dg\Mod(A)$ the category of right $A$-modules in chain complexes. If $A \in dg_+\Alg(R)$, we let $dg_+\Mod(A) \subset dg\Mod(A)$ be the subcategory of objects concentrated in non-negative chain degrees.
\end{definition}

Applying \cite[Theorem 11.3.2]{Hirschhorn} as before, there is a projective model on $dg_+\Mod(A)$ (resp. $dg\Mod(A)$) in which weak equivalences are quasi-isomorphisms and fibrations are surjections in strictly positive degrees (resp. surjections). 

\begin{definition}\label{projmoddef}
For $A \in dg\Alg(R)$, we say that a module $M \in dg_+\Mod(A)$ is homotopy projective if as an object of the homotopy category $\Ho(dg_+\Alg(A))$ it is a direct summand of a  direct sum of copies of $A$, without shifts.
\end{definition}

\begin{lemma}\label{projmodlemma}
 For $C \in dg_+\Alg(R)$, a morphism $f \co M \to N$ in $dg_+\Mod(C )$ has the homotopy left lifting property with respect to all surjections if and only if its cone is homotopy projective.
\end{lemma}
\begin{proof}
This is a standard argument.
By considering homotopy fibre products, it follows that  the morphism $f$ has the homotopy left lifting property with respect to all surjections if and only if its cone does. In other words, this says that for all surjections $K \onto L$  in $dg_+\Mod(C)$, the map 
\[
 \Ext^0_C(\cone(f),K)\to \Ext^0_C(\cone(f),L) 
\]
is surjective. 

This property is satisfied whenever $\cone(f)$ is a direct sum of copies of $C$, and indeed whenever $\cone(f)$ is a retract of such. Conversely, given $\cone(f)$ with the properties above, choosing a set $S$ of generators for $\cone(f)_0$ gives a map $\bigoplus_{s \in S}A \to \cone(f)$ which is surjective in degree $0$. This then leads to a surjection 
\[
\bigoplus_{s \in S}A \oplus\cone(\cone(f)_{>0})_{[1]} \to \cone(f)
\]
from an object of  $dg_+\Mod(C)$ quasi-isomorphic to $\bigoplus_{s \in S}A$, which must admit a homotopy section by hypothesis.
\end{proof}

\subsection{Smooth and \'etale morphisms}

\begin{definition}\label{Omegadef}
 Given  a morphism $f \co C \to A$ in $dg\Alg(R)$, we define the $A^{\oL,e}$-module $\Omega^1_{A/C}$ in complexes to be the kernel of the multiplication map $A\ten_C A \to A$, and we  denote its differential (inherited from $A$) by $\delta$. 

We refer to the left-derived version  $\oL\Omega^1_{A/C}$ of this as the cotangent complex, which is unique up to quasi-isomorphism, and is given by 
\[
 \cocone(A\ten^{\oL}_CA \to A),
\]
again regarded as an $A^{\oL,e}$-module.
\end{definition}

\begin{remarks}\label{cotremarks}
Note that if $C \to A$ is a morphism in $dg_+\Alg(R)$, then  $\oL\Omega^1_{A/C}$   is quasi-isomorphic to 
 $A\ten_{\tilde{A}}\Omega^1_{\tilde{A}/C}\ten_{\tilde{A}}A$ for any factorisation $C \to \tilde{A} \to A$ with $\tilde{A} \to A$ a quasi-isomorphism and $\tilde{A}_{\#}$ flat as a left or right $C_{\#}$-module; in particular this applies if $\tilde{A}$ is a cofibrant replacement of $A$ over $C$. %%dg_+ needed for Spaltenstein purposes.

Observe that for an $A$-bimodule $M$, an $A$-bilinear map $\Omega^1_{A/C} \to M$ is essentially the same thing as an $C$-bilinear derivation $A\to M$, via the universal derivation $d \co A \to \Omega^1_{A/C}$ given by $da=a\ten 1 - 1\ten a$. %%inverse $a\ten b \mapsto adb$.

\end{remarks}

\begin{lemma}\label{cotexact}
 Given morphisms $C  \to B \to A$ in $dg\Alg(R)$, there is a natural exact triangle 
\[
A\ten_B^{\oL}\oL\Omega^1_{B/C}\ten_B^{\oL}A \to \oL\Omega^1_{A/C} \to \oL\Omega^1_{A/B} \to 
\]
of $A^{\oL,e}$-modules.
\end{lemma}
\begin{proof}
Substituting in the definitions, we see that 
\[
 A\ten_B^{\oL}\oL\Omega^1_{B/C}\ten_B^{\oL}A \simeq  \cocone(A\ten^{\oL}_CA \to A\ten^{\oL}_BA),
\]
and that $\cocone(\oL\Omega^1_{A/C} \to \oL\Omega^1_{A/B} )$ reduces to the same expression.
\end{proof}

\begin{lemma}\label{sq0lemma}
Any  surjection $C \to D$ in $dg_+\Alg(R)$ with square-zero kernel $I$ (i.e. $0 = I\cdot I \subset C$)  is quasi-isomorphic  to a diagram of the form
\[
\pr_1 \co D'\by_{\eta, D'\oplus I_{[-1]},0}^hD' \to D', 
\]
for some derivation $\eta \co D' \to I_{[-1]}$, where $\by^h$ denotes the homotopy fibre product.
\end{lemma}
\begin{proof}
 This follows by adapting the argument of for instance \cite[Proposition 1.17%\ref{drep-obs}
 ]{drep} to the setting of DGAAs.
\end{proof}

\begin{lemma}\label{QIMlemma}
 A morphism $f \co A \to B$ in  $dg_+\Alg(R)$ is a quasi-isomorphism if and only if
\begin{enumerate}
 \item the map  $\H_0f \co \H_0A \to \H_0B$ is an isomorphism, and
\item the relative cotangent complex $\oL\Omega^1_{B/A}\ten^{\oL}_{B^{\oL,e}}(\H_0B\ten_{\H_0R}\H_0B)$ is acyclic.
\end{enumerate}
\end{lemma}
\begin{proof}
 For the ``only if'' direction, observe that a quasi-isomorphism automatically induces an isomorphism on $\H_0$,  and that it makes the relative cotangent complex $\oL\Omega^1_{B/A} $ acyclic because the spaces of derived derivations must agree.

For the ``if'' direction, consider the transformation $f^* \co \map_{dg_+\Alg(R)}(B,-) \to \map_{dg_+\Alg(R)}(A,-)$ of derived mapping spaces. Since  $\H_0f$ is an isomorphism, the transformation $f^*$ is a weak equivalence when evaluated on any $C \in \Alg(\H_0R)$.  For arbitrary $C \in dg\Alg_+(R)$, we proceed by induction on the Postnikov tower of $C$, i.e. the system of quotients $P_nC :=C/\tau_{>n}C $ by good truncations. The $n=0$ case follows because $P_0C \simeq \H_0C$. 

For the inductive step, we use the fact that the  map $P_nC \to P_{n-1}C$ is homotopy equivalent to  a surjection  with square-zero kernel $(\H_nC)_{[-n]}$. Lemma \ref{sq0lemma} thus  expresses $\map_{dg_+\Alg(R)}(B,P_nC)$ as the homotopy fibre product of a diagram
\[
\xymatrix@R=0ex{
\map_{dg_+\Alg(R)}(B,P_{n-1}C)\ar[dr]^-{\eta} \\
 &\map_{dg_+\Alg(R)}(B,H_0C \oplus (\H_nC)_{[-n-1]}),\\
\map_{dg_+\Alg(R)}(B,H_0C)\ar[ur]^-{0}
}
\]
and similarly for $A$. 

Over a fixed map $g \co B \to \H_0C$, the homotopy fibre of $ \map_{dg_+\Alg(R)}(B,H_0C \oplus I) \to \map_{dg_+\Alg(R)}(B,H_0C)$ has $i$th homotopy group 
\[
 \Ext^{-i}_{B^{\oL,e}}(\oL\Omega^1_B, I)= \Ext^{-i}_{(\H_0B\ten_{\H_0R}\H_0B) }( \oL\Omega^1_{B}\ten^{\oL}_{B^{\oL,e}}(\H_0B\ten_{\H_0R}\H_0B),I),
\]
which is isomorphic to the corresponding expression for $A$, by Lemma \ref{cotexact} and acyclicity of $ \oL\Omega^1_{B/A}\ten^{\oL}_{B^{\oL,e}}(\H_0B\ten_{\H_0R}\H_0B)$. Substituted in the homotopy fibre product above and combined with the inductive hypothesis, this yields the required equivalence $ \map_{dg_+\Alg(R)}(B,P_nC) \simeq  \map_{dg_+\Alg(R)}(A,P_nC) $. 

Passing to the homotopy limit over $n$ then gives $\map_{dg_+\Alg(R)}(B,C) \simeq  \map_{dg_+\Alg(R)}(A,C)$, which implies that $f \co A \to B$ is a weak equivalence.
\end{proof}

%%have cut strongness result from here

\begin{definition}%%issue that we're effectively taking retract here, so l.f.p. rather than finitely presented.
We say that a morphism $f \co A \to B$ in $\Alg(\H_0R)$ is l.f.p. if for any filtered system $\{C_i\}_i$ in $\Alg(H_0R)$ with colimit $C$, the morphism
\[
 \LLim_i\Hom_{\Alg(H_0R)}(B,C_i) \to  \LLim_i \Hom_{\Alg(H_0R)}(A,C_i)\by^h_{\Hom_{\Alg(H_0R)}(A,C)}\Hom_{\Alg(RH_0)}(B,C)
\]
of $\Hom$-sets is an isomorphism.

 We say that a morphism $f \co A \to B$ in $dg\Alg(R)$ is homotopy l.f.p. if for any filtered system $\{C_i\}_i$ in $dg\Alg(R)$ with colimit $C$, the morphism
\[
 \LLim_i\map_{dg\Alg(R)}(B,C_i) \to  \LLim_i \map_{dg\Alg(R)}(A,C_i)\by^h_{\map_{dg\Alg(R)}(A,C)}\map_{dg\Alg(R)}(B,C)
\]
of derived mapping spaces is a weak equivalence. 
\end{definition}

In particular, this means that if a morphism $f$ in $dg_+\Alg(R)$ is  homotopy l.f.p., then $\H_0f$ is l.f.p., but beware that l.f.p. morphisms in $\Alg(R)$ are not homotopy  l.f.p. in general.

\begin{lemma}\label{hlfplemma}
 A morphism $A \to B$ in  $dg\Alg_+(R)$ is homotopy l.f.p. if and only if
\begin{enumerate}
 \item the morphism $\H_0A \to \H_0B$ in $\Alg(\H_0R)$ is l.f.p., and 
\item the relative cotangent complex %%$\oL\Omega^1_{B/A}$ is perfect as a $B^{\oL,e}$-module.
 $\oL\Omega^1_{B/A}\ten^{\oL}_{B^{\oL,e}}(\H_0B\ten_{H_0R}\H_0B)$ is perfect as an $\H_0B\ten_{H_0R}\H_0B$-module. 
\end{enumerate}
\end{lemma}
\begin{proof}
 The ``only if'' direction is immediate; we prove the ``if'' direction by a similar induction to that used in Lemma \ref{QIMlemma}. Since $\oL\Omega^1_{B/A} \ten^{\oL}_{B^{\oL,e}}(\H_0B\ten_{H_0R}\H_0B)$ is perfect, it has finite projective dimension, $k$ say.  

For any  filtered system $\{C(i)\}$ in  $dg_+\Alg(R)$ with $C:= \LLim_{i \in I} C(i)$, fix an element  $f_l$ of the derived mapping space $\map(A,C_l)$ for some $l \in I$; denote its images in $\map(A,C(i))$ (for $i>l$) and in $\map(A,C)$ by $f_i$ and $f$ respectively. Writing $\map(B,C)_f$ (resp. $\map(B,C(i))_{f_i}$) for the homotopy fibre of $\map(B,C) \to \map(A,C)$ over $f$ (resp.  $\map(B,C(i)) \to \map(A,C(i))$ over $f_i$, the desired statement amounts to showing that the natural map
\[
 \LLim_{i \ge l} \map(B,C(i))_{f_i} \to \map(B,C)_f
\]
is a weak equivalence.

For the Postnikov tower $P_n$ introduced in Lemma \ref{QIMlemma},  denote by $P_nf$ (resp. $P_nf_i$)  the image of $f$ (resp. $f_i$) in $ \map(A,P_nC)$ (resp. $ \map(A,P_nC(i))$). For $n=0$, the first hypothesis is equivalent to saying that the morphism
\[
  \LLim_i\map(A,P_0C(i))_{P_0f_i} \to  \map(A,P_0C)_{P_0f}
\]
of homotopy fibres is a weak  equivalence.
 Applying 
Lemma \ref{sq0lemma} and taking homotopy fibres  gives the expression
\[
 \map(B,P_nC)_{P_nf} \simeq  \map(B,P_{n-1}C)_{P_{n-1}f}\by^h_{\eta, \map(B,H_0C \oplus (\H_nC)_{[-n-1]})_{P_0f},0 }\map(B,H_0C)_{P_0f},
\]
and similarly for each $(C(i),f_i)$. Since the homotopy fibres of $\map(B,H_0C \oplus (\H_nC)_{[-n-1]})_{P_0f} \to \map(B,H_0C) $ are governed by the relative cotangent complex, which is perfect,  they commute with filtered colimits in $C$.  Applied inductively, this shows that the morphisms
\[
  \LLim_i\map(A,P_nC(i))_{P_nf_i} \to  \map(A,P_nC)_{P_nf}
\]
are weak equivalences for all $n$.

It remains to show that we can pass the filtered colimit in $i$ past  the homotopy limit in $n$. To do this, we look at the behaviour of the homotopy groups at each stage in the induction. Let $d$ be the projective dimension of   $L:= \oL\Omega^1_{B/A}\ten^{\oL}_{B^{\oL,e}}(\H_0B\ten_{H_0R}\H_0B)$, which is finite because the complex is perfect. We have  exact sequences
\begin{align*}
 \Ext^{n-k}_{\H_0B\ten_{H_0R}\H_0B}(L, \H_nC) \to &\pi_k\map(B,P_nC)_{P_nf}\\
 &\to \pi_k\map(B,P_{n-1}C)_{P_{n-1}f} \to \Ext^{n-k+1}_{\H_0B\ten_{H_0R}\H_0B}(L, \H_nC),
\end{align*}
from which we deduce that the maps $\pi_k \map(B,P_nC)_{P_nf} \to \pi_k\map(B,P_{n-1}C)_{P_{n-1}f}$ are isomorphisms for all $n>k+d$, and similarly for $C(i)$. 

In particular, this means that the derived limits $\Lim^1$ of these homotopy groups vanish, and hence that
\[
 \pi_k\map(B,C)_{f} \cong \Lim_n  \pi_k \map(B,P_nC)_{P_nf} \cong   \pi_k\map(B,P_{k+d+1}C)_{P_{k+d+1}f},
\]
and similarly for $C(i)$.  Substituting the result for $P_{k+d+1}C$ above, we therefore conclude that the morphisms
\[
  \LLim_{i \ge l} \pi_k\map(B,C(i))_{f_i} \to \pi_k\map(B,C)_f
\]
 are all isomorphisms, as required.
\end{proof}

\begin{definition}\label{smoothetdef} 
We say that a  morphism $f \co A \to B$ in $\Alg(\H_0R)$ is  formally submersive (resp. formally \'etale) if  for any surjection $C \to D$ in $\Alg(R)$ with nilpotent kernel, the morphism
\[
 \Hom_{\Alg(R)}(B,C) \to \Hom_{\Alg(R)}(B,D)\by_{\Hom_{\Alg(R)}(A,D)}\Hom_{\Alg(R)}(A,C)
\]
of $\Hom$-sets is surjective  (resp. an isomorphism). We then say that $f$ is submersive (resp.  \'etale) it is  also l.f.p.

 We say that a  morphism $f \co A \to B$ in $dg_+\Alg(R)$ is homotopy formally submersive (resp. homotopy formally \'etale) if  for any surjection $C \to D$ in $dg_+\Alg(R)$ with nilpotent kernel, the morphism
\[
 \map_{dg_+\Alg(R)}(B,C) \to \map_{dg_+\Alg(R)}(B,D)\by^h_{\map_{dg_+\Alg(R)}(A,D)}\map_{dg_+\Alg(R)}(A,C)
\]
of mapping spaces is surjective on $\pi_0$ (resp. a weak equivalence). We   then say that $f$ is homotopy  submersive (resp. homotopy \'etale) it is  also homotopy l.f.p. 
\end{definition}

\begin{remark}
 The concepts of submersiveness and \'etaleness are chosen so that if a morphism $A \to B$ of DGAAs is homotopy submersive (resp. homotopy \'etale), then the morphism $A/([A,A]) \to B/([B,B])$ of abelianisations is a homotopy smooth (resp. homotopy \'etale) morphism of CDGAs. The reason for our use of ``submersive'' instead of ``smooth'' is to avoid a terminological conflict with the notion of smoothness from  \cite{KontsevichSoibelmanAinftyalgcats}, which was chosen to ensure that the forgetful functor from CDGAs to DGAAs preserves smoothness; every homotopy l.f.p. DGAA is smooth in that sense, since its cotangent complex is perfect as a bimodule.
\end{remark}

\begin{lemma}\label{smoothetlemma}
 A morphism $f \co A \to B$ in $dg_+\Alg(R)$ is homotopy \'etale (resp, homotopy submersive) if $f$ is homotopy l.f.p. and the $B^{\oL,e}$-module
$\oL\Omega^1_{B/A}$ is acyclic (resp. homotopy projective).
\end{lemma}
\begin{proof}
This is a standard obstruction-theoretic argument. By induction, it suffices to consider surjections $C \to D$ with square-zero kernel in Definition \ref{smoothetdef}. Applying Lemma \ref{sq0lemma} then reduces the question to square-zero extensions of the form $D \oplus M \to D$. The statement now follows from Lemma \ref{projmodlemma}  because for $f \in \map_{dg_+\Alg(R)}(A,D)$, we have
\[
 \map_{dg_+\Alg(R)}(A,D \oplus M)\by^h_{\map_{dg_+\Alg(R)}(A,D)}\{f\} \simeq \map_{dg_+\Mod(A^{\oL,e})}(\oL\Omega^1_A,f_*M), 
\]
and similarly for $B$.
\end{proof}

\begin{definition}\label{NCprestdef}
 The model category $\Aff^{nc}(R)^{\wedge}$ of  NC prestacks  is defined to be the category of simplicial set-valued functors on $\Alg(R)$, equipped with the  projective model structure 

 The model category $DG^+\Aff^{nc}(R)^{\wedge}$ of derived NC prestacks  is defined to be the category of simplicial set-valued functors on $dg_+\Alg(R)$, equipped with the left Bousfield localisation of the projective model structure at morphisms of the form $f^* \co \oR \Spec^{nc} B \to \oR\Spec^{nc} A$, for quasi-isomorphisms $A \to B$.
\end{definition}
As in \cite[\S 2.3.2]{hag1},%%reiterated after definition 4.1
the homotopy category $\Ho(DG\Aff^{nc}(R)^{\wedge})$ is equivalent to the category of weak equivalence classes of weak equivalence-preserving simplicial functors on $dg_+\Alg(R)$, since fibrant objects are those prestacks which preserve weak equivalences and are objectwise fibrant.

\begin{definition}
Given $A \in \Alg(\H_0R)$, define the presheaf $\Spec^{nc} A \in \Aff^{nc}(\H_0R)^{\wedge}$ to be the functor $\Hom_{\Alg(\H_0R)}(A,-)$ of homomorphisms.

 Given $A \in dg_+\Alg(R)$, define the simplicial presheaf $\oR\Spec^{nc} A \in  DG\Aff^{nc}(R)^{\wedge}$ to be the derived mapping functor $\map_{dg_+\Alg(R)}(A,-)$. 
\end{definition}

\begin{remark}
Derived mapping functors can be calculated using simplicial framings as in \cite[Theorem 5.4.9]{Hovey} by setting $\map:= \oR\Map_r$. If $A$ is cofibrant (i.e. if $A_{\#}$ is a retract of a free graded associative algebra) and $\Q \subset R$, then one  such explicit model for $\map_{dg_+\Alg(R)}(A,B)$ is given by the simplicial set 
\[
 n \mapsto \Hom_{dg\Alg(R)}(A,B\ten_{\Q}\Omega^{\bt}(\Delta^n)),
\]
where $\Omega^{\bt}(\Delta^n):= \Q[t_0, t_1, \ldots, t_n,\delta t_0, \delta t_1, \ldots, \delta t_n ]/(\sum t_i -1, \sum \delta t_i)$
is the CDGA of de Rham polynomial forms on the $n$-simplex, with the $t_i$ of degree $0$.
\end{remark}
\begin{definition}
 Say that a morphism $\oR \Spec^{nc} B \to \oR\Spec^{nc} A$ is a closed immersion if the associated map $\H_0A \to \H_0B$ is surjective.
\end{definition}

\begin{definition}
Say that a derived NC prestack  is affine, or  a derived affine NC scheme, if it is weakly equivalent to one of the form $\oR \Spec^{nc} A$, for $A \in dg_+\Alg(R)$. %We then say that an NC prestack is striclty ind-affine, or  a derived strictly ind-affine NC scheme,  if it is weakly equivalent to  a colimit of a filtered diagram of closed immersions between derived affine NC schemes.
\end{definition}

\subsection{Representation spaces and related constructions} %{Commutative restrictions, restriction of scalars and Hall algebras}

\begin{definition}\label{commdef}
If $\Q \subset R$, then 
 given an NC prestack (resp.   derived NC prestack) $F$  over $R$, denote by $F^{\comm}$ the restriction of $F$ to commutative $R$-algebras $C\Alg(R)$ (resp. CDGAs  $dg_+C\Alg(R)$).
\end{definition}

Note that if $F=\Spec^{nc} A$ (resp. $\oR\Spec^{nc} A$)  is an NC affine scheme (resp. a derived NC affine scheme), then $F^{\comm}$ is an  affine scheme (resp. a derived affine scheme), given by $\Spec (A/([A,A]))$ (resp. $\oR\Spec (A'/([A',A']))$ for a cofibrant replacement $A'$ of $A$).

\subsubsection{Restriction of scalars}\label{weilrestrn}

Given an associative $R$-algebra $S$, there is an endofunctor $\Pi_{S/R}$ of $ \Aff^{nc}(R)^{\wedge}$  given by sending a prestack $F$ to the prestack $F(-\ten_RS)$. We can think of this as a form of Weil restriction of scalars, or as a mapping prestack  (of maps from $\Spec^{nc} S$ to $F$). When no ambiguity is likely, we simply denote $\Pi_{S/R}$ by $\Pi_S$.

The following is an immediate consequence of the general adjoint functor theorem and exactness of flat tensor products:
\begin{lemma}\label{weillemma1}
If $S$ is finite and flat as an $R$-module, then the  restriction of scalars functor $\prod_{S/R}$ on $ \Aff^{nc}(R)^{\wedge}$  preserves the affine NC schemes  $ \Aff^{nc}(R)$. This functor preserves submersive morphisms, \'etale morphisms, epimorphisms and open morphisms. 
\end{lemma}

\begin{example}\label{repnex}
By  taking $S$ to be the ring $\Mat_n(R)$ of $n \by n$ matrices,  for any NC affine scheme $\Spec^{nc} A$ over $R$, the NC affine scheme $ \Pi_{\Mat_n/R}\Spec^{nc} A$ represents the functor of $R$-algebra homomorphisms from $A$, so
\[
(\Pi_{\Mat_n/R}\Spec^{nc} A)(B)= \Hom_{\Alg(R)}(A, \Mat_n(B)).
\]
Thus $ (\Pi_{\Mat_n/R}\Spec^{nc} A)^{\comm} $ is the affine scheme representing framed $n$-dimensional representations of $A$. %%for instance $A=R[G]$ for a discrete group $G$.

We can also consider the NC prestack $[\Spec^{nc} A/\bG_m] $, which sends $B$ to the nerve of the groupoid $[\Hom_{\Alg(R)}(A,B)/B^{\by}]$ of $R$-algebra homomorphisms modulo inner automorphisms, where $\bG_m(B):=B^{\by}$ acts by conjugation. Then the hypersheafification of 
\[
(\Pi_{\Mat_n}[\Spec^{nc} A/\bG_m])^{\comm} 
\]
is the Artin $1$-stack representing $n$-dimensional representations of $A$.
\end{example}

There is similarly an endofunctor $\Pi_{S/R}$ of $ DG^+\Aff^{nc}(R)^{\wedge}$  given by sending a prestack $F$ to the prestack $F(-\ten_RS)$. Again, we can think of this as Weil restriction of scalars, or as a mapping prestack. Similarly, we have:
\begin{lemma}\label{weillemma2}
If $S\in \Alg(R)$ is finite and flat as an $R$-module, then the  restriction of scalars functor $\prod_{S/R}$ on $ DG^+\Aff^{nc}(R)^{\wedge}$ preserves the derived affine NC schemes  $ DG^+\Aff^{nc}(R)$. This functor preserves submersive morphisms, \'etale morphisms, epimorphisms and open morphisms.  %% If $B \onto C$, $B\ten S \onto C\ten S$.
\end{lemma}%%won't generalise to perfect $R$-mods: just take $F=\bA^1$. 

% As an immediate consequence, we have: 
% \begin{lemma}\label{weillemma}
% For any $n$-geometric NC Artin prestack $F$, the functor $F(-\ten_RB)$ is also an $n$-geometric NC Artin prestack $F$.
% \end{lemma}
% 
% Substituting in Propostion \ref{projgeomprop} and Corollary  \ref{Perfgeomcor} then gives:
% \begin{corollary}
% The NC prestack $B\cP(-\ten_RB)$ of finite projective modules is  submersive Artin $1$-geometric, and the    NC prestack $\Perf(-\ten_RB)$ is  Artin $\infty$-geometric.
% \end{corollary}
% 

\subsubsection{Flags and Hall algebras}

If we let $T(r_1,\ldots,r_n)$ be the ring of block upper triangular matrices with diagonal blocks of size $r_1, \ldots,r_n$, then the restriction of scalars 
\[
(\Pi_{ T(r_1,\ldots,r_n)}[\Spec^{nc} A/\bG_m])^{\comm} 
\]
gives the groupoid of flags of $A$-representations of signature $(r_1,\ldots,r_n)$, by adapting Example \ref{repnex}. The graded quotients of the flag  are induced by  the natural ring homomorphisms 
\[
p_i \co T(r_1,\ldots,r_n) \to \Mat_{r_i},
\]
while the inclusion 
\[
\iota \co T(r_1,\ldots,r_n) \to \Mat_{\sum_i r_i}
\]
sends a flag to the underlying representation.

%%For things below, the equivce $\Pi_{M_1} \to\Pi_{M_n}$ has inverse functor $\ten_{M_n(B)}B^n$, since get $(M,0,\ldots,0) \in \Mod(B^n)$, base changed to $(M,0,\ldots,0)\ten_{B^n}\M_n(B) \cong \bigoplus_n (M,0,\ldots,0)\ten_{B^n}c_n= (M,0,\ldots,0)\ten_{B^n}(c\ten r)$ (for $c$ a row and $r$ a column)  and $M\ten_B c$ is a copy of $M$, but then base change gives $r\ten_{\Mat_n}B^n \cong B$, since $(c\ten r) \ten_{\Mat_n}B^n\cong B^n$. Need to think about inverse functors for $T$. $\Pi_{T(r,r,\ldots,r)} \to \Pi_{T_n}$ shd work similarly as $\ten_{T(r,\ldots)}(B^r\ten B^n)$, but tricky otherwise. $\End_{\Mat_n}(B^n) \cong B$ (acting diagonally on other side). What is $\End_{T(r_1,\ldots,r_n)(\bigoplus_i B^{R_i})$.  

A similar phenomenon arises for the moduli functors $\sP$  and $\Perf$ of projective modules and of perfect complexes in \S\S \ref{modprojmodsn} and \ref{modperfsn} below. For those functors $F$, we have natural equivalences $\Pi_{\Mat_n}F \simeq F$ and $ \Pi_{T(r_1,\ldots,r_n)}F \simeq \Pi_{T_n}F$, where $T_n:= T(\underbrace{1,\ldots,1}_n)$. %%given by $\ten_{\Mat^n(R)}R^n$ and $\ten_{T(r_1,\ldots,r_n)\Hom_{\Fil}(\bigoplus_i R^{r_i}, \bigoplus_i R)$.
The motivic Hall algebra product induced on $F^{\comm}$ (see for instance \cite[\S 4.2]{bridgelandMotivicHall})
thus comes from push-pull homomorphisms associated to the diagram  
\[
F \simeq  \Pi_{\Mat_2} F \xla{\iota}  \Pi_{T_2}F  \xra{(p_1,p_2)} F \by F.
\]

We now summarise the conditions on an NC prestack $F$ which allow us to construct a Hall algebra in this way. To begin with, we need   a marked basepoint $\{0\} \in F(R)$ and we need $F$ to commute with finite products. 

We will also need $F$ to be stable under adding $0$ in the following sense. We have canonical composite maps 
\begin{align*}
\Pi_{M_n}F \xra{(\id, 0)}   (\Pi_{M_n}F) \by F \cong \Pi_{M_n\by M_1}F &\to  \Pi_{M_{n+1}}F\\
\Pi_{T(r_1,\ldots,r_n)}F \xra{(\id, 0)}   (\Pi_{T(r_1,\ldots,r_n)}F) \by F \cong \Pi_{T(r_1,\ldots,r_n)\by M_1}F &\to  \Pi_{T(r_1,\ldots,r_{i-1}, r_i+1, r_{i+1}, \ldots r_n)}F,
\end{align*}
given by insertion of $1 \by 1$-matrices $M_1$ on the diagonal, and we require  these maps to be equivalences and independent of the $n+1$ (resp. $r_i+1$) choices of position to insert the new entry. %%or just ask for map the other way.

We then have a simplicial diagram of prestacks given in degree $n$ by $\Pi_{T_n}F \simeq \Pi_{T(r_1,\ldots,r_n)}F$, with simplicial operations induced by the obvious maps 
\[
\pd_i \co T(r_1, \ldots,r_n) \to \begin{cases} T(r_2, \ldots,r_n) & i=0\\
                                  T(r_1, \ldots,r_{i-1}, r_i+r_{i+1}, r_{i+2}, \ldots, r_n) & 1 \le i \le n-1 \\
T(r_1, \ldots,r_{n-1}) & i=r,
                                 \end{cases}
\]
 and
\[
 \sigma_i \co \Pi_{T(r_1, \ldots,r_n)}F \xra{(\id,0)} \Pi_{T(r_1, \ldots,r_n) \by M_1}F   \to T(r_1, \ldots, r_i, 1, r_{i+1}  r_n)F.
\]

For this to result  in an associative Hall algebra product amounts to requiring that $\Pi_{T_{\bt}}F $ be a $2$-Segal object in the sense of \cite{DyckerhoffKapranovHigherSegal}. 
Substituting in \cite[Theorem 2.3.2]{DyckerhoffKapranovHigherSegal}, it suffices to show that that for $1 \le j\le n$,   the natural diagram
\[
\begin{CD}
 \Pi_{T(r_0,r_1, \ldots, r_n)}F @>>> \Pi_{T(r_0,r_1+\ldots+r_j,r_{j+1}, \ldots,r_n)}F \\
 @VVV @VVV\\
 \Pi_{T(r_0,r_1,  \ldots, r_j)}F @>>>\Pi_{T(r_0,r_1+\ldots+r_j)}F
 \end{CD}
\]
is a homotopy pullback square. Since we are assuming that $F$ preserves products, we may rewrite 
the homotopy fibre product as
\[
 \Pi_{T(r_0,r_1+\ldots+r_j,r_{j+1}, \ldots,r_n)}F \by^h_{\Pi_{T(r_0,r_1+\ldots+r_j) \by M_{r_{j+1}} \by \ldots \by M_{r_n} }F}\Pi_{T(r_0,r_1,  \ldots, r_j)\by M_{r_{j+1}} \by \ldots \by M_{r_n} }F
\]
If $F$ is homogeneous in the sense of Definition \ref{hhgsdef} below, we may then exploit nilpotence of the surjection 
\[
T(r_0,r_1+\ldots+r_j,r_{j+1}, \ldots,r_n) \to  T(r_0,r_1+\ldots+r_j) \by M_{r_{j+1}} \by \ldots \by M_{r_n}
\]
to deduce that 
\begin{align*}
 &F(T(r_0,r_1+\ldots+r_j,r_{j+1}, \ldots,r_n)(B))\by^h_{F(T(r_0,r_1+\ldots+r_j)(B)) %\by M_{r_{j+1}}(B) \by \ldots \by M_{r_n}(B)
 }F(T(r_0,r_1,  \ldots, r_j)(B))%\by M_{r_{j+1}}(B) \by \ldots \by M_{r_n}(B)
 \\
 &\simeq F(T(r_0,r_1+\ldots+r_j,r_{j+1}, \ldots,r_n)(B)\by_{T(r_0,r_1+\ldots+r_j)(B)}T(r_0,r_1,  \ldots, r_j)(B))\\
 &= F(T(r_0,r_1, \ldots, r_n)(B)),
\end{align*}
as required.

Since \cite[\S 3.2]{DyckerhoffKapranovHigherSegal} allows us to associate Hall algebras to $2$-Segal objects, this means that there are Hall algebras associated to any  NC prestack $F$ which is product-preserving, homogeneous and has a basepoint $\{0\}\in F(R)$ for which $F$ is stable under adding $0$ in the sense above.

\subsection{Comparisons with other non-commutative spaces}

\subsubsection{From pre-triangulated dg categories to NC prestacks}\label{dgcatsn1}

One perspective on non-commutative geometry, appearing for instance in \cite{KatzarkovKontsevichPantevHodgeMirror,OrlovSmoothProperNC} is that 
a  NC space is a pre-triangulated dg category, with schemes and algebraic stacks being replaced by their dg categories of perfect complexes.

Given such a dg category $\cA$, there are two candidates for the most natural derived NC prestack which we can associate to it. One is the functor of perfect complexes $\Perf_{\cA}\co B \mapsto \Perf(\cA\ten_R^{\oL}B)$ (see Definition \ref{Perfdef} below for a precise definition of $\Perf$).  
The other 
sends $B$ to the space of derived Morita morphisms from $\cA$ to $B$, i.e. the space of dg functors from $\cA$ to the dg category $\per_{dg}(B)$ of perfect $B$-modules; by \cite[Theorem 1.1]{toenMorita}, %\cite[Corollary 5.1]{tabuadaInvariantsAdditifs} 
  this is equivalent to the nerve $\Mor(\cA,B)$ of the core of the simplicial category  associated to the dg category $\mathrm{mor}_{dg}(\cA,B)$ of those $\cA-B$-bimodules which are perfect over $B$. When $\cA$ is a smooth and proper dg category over $R$, observe that we have a natural equivalence $\Mor(\cA,B) \simeq \Perf(\cA^{\op}\ten_R^{\oL}B)$.
%
%% \cite{kaledinBeilinson} observes that if $\cA$ smooth and proper, then these are the same. properness clearly gives map $\Perf(\cA^{\op} \to \Ps\Perf(\cA)$. For converse, note that $F \co \cA \to \per(B)$ gives $F \co \cA^{\oL,e} \to \cA^{\op} \ten\per(B)$, and hence if $\cA$ is perfect over $\cA^{\oL,e}$, $F$ sends it to sth perfect in $\cA^{\op} \ten\per(B)$
%%

There are also much smaller NC prestacks we can consider, such as the functor  of finite projective $\cA-B$-bimodules. When $A \in \Alg(R)$, a natural \'etale subfunctor of $\pi^0\Mor$ is given by the prestack $[\oR \Spec^{nc} A/\bG_m]$ introduced in \S \ref{weilrestrn}.
% , which sends $B$ to the groupoid $[\Hom_{\Alg(R)}(A,B)/B^{\by}]$ of $R$-algebra homomorphisms modulo inner automorphisms, where $\bG_m(B):=B^{\by}$ acts by conjugation. 
%\Perf_{\cA}
%Open subthing of Morita then given by line bundles, hence a copy of $[\Spec^{nc} A/\bG_m]$. --- that's of relevance to things like \cite{GinzburgSchedler}, the main point being that it exists and that inner constructions on algebras are for us just constructions on that stack.
%

\subsubsection{From NC prestacks to quasi-NC enhancements}\label{quasiNCsn}

In \cite{kapranovNCGeomCommutator}, Kapranov introduced a notion of NC schemes (referred to as quasi-NC structures in \cite{todaNCModuliStableSheaf}), consisting of a ringed space $(X,\sO_X)$ where $\sO_X$ is associative, its abelianisation  $\sO_X^{\comm}$ defines a scheme $(X,\sO_X^{\comm})$, and the map $\sO_X \to\sO_X^{\comm}$ is pro-nilpotent. Thus $X$ is a  non-commutative nilpotent thickening of  a scheme. Given  an NC prestack $F$, we can extend $F^{\comm}$ to a functor on the category of Kapranov NC schemes $(X,\sO_X)$ by setting
\[
 \map(X,F):= \ho\Lim_d \oR\Gamma(X, F(\sO_X^{\le d})),
\]
where $\sO_X =\Lim_d\sO_X^{\le d}$. In particular, for an $n$-geometric NC prestack $F$ as in \S \ref{NCstacksn} below, this  will yield a quasi-NC enhancement of the $n$-geometric Artin stack $(F^{\comm})^{\sharp}$ on completion and sheafification.

Toda's quasi-NCDG structures from \cite{todaNCModuliStableSheaf} are a dg enhancement of Kapranov's quasi-NC structures. For any derived NC prestack $F$, we can extend $F^{\comm}$ to a functor on the category of Toda's quasi-NCDG structures by the formula above. When applied to an  $n$-geometric derived  NC prestack $F$, this  will similarly yield a quasi-NCDG enhancement of the $n$-geometric derived Artin stack $(F^{\comm})^{\sharp}$.  

It is important to note that a (derived) NC prestack $F$ leads to such enhancements not just of the commutative restriction $F^{\comm}$, but also (via \S \ref{weilrestrn}) of the representation spaces $(\Pi_{\Mat_r}F)^{\comm}$.

\subsubsection{From NC prestacks to  non-commutative deformations}

Given a field $k$ and a point $x \in F(k)$ of a (derived) NC prestack, we can study the infinitesimal neighbourhood of $x$ by looking at the functor $\hat{F}_x \co A \mapsto F(A)\by^h_{F(x)}\{x\}$ for local (dg) Artinian algebras $A$ with residue field $k$, giving a simplicial set-valued functor we we can think of as the $\infty$-groupoid of (derived) NC deformations of $x$. Note that in the derived case, the functor $\hat{F}_x$ is only defined on  Artinian objects $A \in dg_+\Alg(R)$,  but that the non-commutative analogue of \cite[Corollary 4.49]{ddt1} gives a natural extension to Artinian objects in $dg \Alg(R)$ whenever $F$ is homogeneous in the sense of Definition \ref{hhgsdef} below. 

Truncating to the fundamental groupoid then gives us a groupoid-valued functor $\pi_f\hat{F}_x$ on Artinian DGAAs, which is a  deformation functor in the sense of  \cite{EfimovLuntsOrlov1}. In particular, both of the functors $ \pi_f\widehat{(\Perf_{\cA})}_E$ and $ \pi_f\widehat{\Mor(\cA,-)}_E$ induced by the prestacks discussed in \S \ref{dgcatsn1} correspond to the functor $\ddef^h(E)$ from  \cite{EfimovLuntsOrlov1}. 

% \cite{DonovanWemyss}, \cite{boothContractionAlg}, \cite{HuaKellerClusterCats}. %%%many of those belong in NCpoisson, 

However, as observed by Kawamata in \cite{kawamataMultipointedNCDefsCY3}, non-commutative deformation theory is not local in nature, because a non-commutative Artin semi-local algebra with nilpotent Jacobson radical is not usually a product of Artin local algebras. A simple manifestation of this phenomenon is that non-commutative deformations of an $R$-module $M$ are governed by the DGAA  $\oR\EEnd_R(M)$, so deformations of the $R\by R$-module $M \by N$ are governed by the DGAA $\oR\EEnd_{R\by R}(M \by N) \neq \oR\EEnd_{R}(M) \by  \oR\EEnd_{R}(N)$. Thus the sheafification inherent to the quasi-NC structures of \S \ref{quasiNCsn}  destroys information about distant commutative points interacting with each other non-commutatively.

Given a (derived) NC prestack $F$, a semisimple algebra $S$ and a point $x \in F(S)$, there is an associated  functor $\hat{F}_x \co A \mapsto F(A)\by_{F(S)}\{x\}$ on the category of nilpotent extensions $A$ of $S$, in other words a (derived) NC deformation functor, regarded as multi-pointed when $S$ is not simple. In contrast with the commutative setting, if $S$ decomposes into simple algebras as $\prod_i S_i$,  with $x_i \in F(S_i)$ the image of $x$, then $ \hat{F}_x$ cannot be recovered from the individual deformation functors $\hat{F}_{x_i}$.

\section{Non-commutative analogues of algebraic stacks}\label{NCstacksn}

\subsection{Non-commutative $n$-geometric prestacks}

We now set about forming non-commutative enhancements of  moduli stacks, using our non-commutative affine schemes as the building blocks. In the non-commutative setting, the typical difficulty with trying to glue  affine schemes or to take stacky quotients  is that descent does not behave well. We bypass this problem by observing that for our motivating examples, there exist  atlases on the level of prestacks, without needing descent. In other words, the moduli functors $F$ 
admit filtered systems $\{U_i \to F\}_i$ of submersive morphisms from non-commutative affines, such that $\LLim_i U_i \to F$ is a surjection of \emph{presheaves}. %%here or later, mention that we don't take coproduct because not representable as presheaves, and unnecessarily messy oterhwise. Or should we try imposint sth like cdn $F(A \by B)= F(A)\by F(B)$? No point.

We now fix a commutative ring $R$, which will act as our base ring.

\subsubsection{Definitions}

The following definition arises by substituting our context in \cite[Definitions \ref{stacks2-npreldef} and \ref{stacks2-nptreldef}]{stacks2}. %Note that an epimorphism $F \to G$ of prestacks is a morphism for which the maps $\pi_0F(A) \to \pi_0G(A)$ are surjective for all $A$. 
For an explicit description of the partial matching objects featuring here, see \cite[\S 2]{stacksintro}.

\begin{definition}\label{npreldef}
Given  a simplicial diagram  $Y\in \Aff^{nc}(R)^{\Delta^{\op}}$, define an  NC Artin (resp. NC Deligne--Mumford) $n$-hypergroupoid  over $Y$ to be a morphism $X_{\bt}\to Y_{\bt}$ in $ \Aff^{nc}(R)^{\Delta^{\op}}$, for which  the   partial matching maps
\[
X_m \to \Hom_{s\Set}(\L^m_k, X)\by_{ \Hom_{s\Set}(\L^m_k, X) }^hY_m 
\]
are   submersive (resp. \'etale) epimorphisms for all  $m\ge 1$, and are isomorphisms for all $m>n$. %and all $k\ge 0$ 
When $Y= \Spec R$ is the final object, we will simply refer to $X$ as an NC Artin (resp. NC Deligne--Mumford)  $n$-hypergroupoid.

Define a  trivial  NC Artin (resp. NC Deligne--Mumford) $n$-hypergroupoid  over $Y$ to be a morphism $X_{\bt}\to Y_{\bt}$ in $ \Aff^{nc}(R)^{\Delta^{\op}}$, for which 
the  matching maps
\[
X_m \to  \Hom_{s\Set}( \pd \Delta^m, X)\by_{\Hom_{s\Set}( \pd \Delta^m, Y)  }Y_m 
\]
are  submersive (resp. \'etale) epimorphisms for all $m \ge 0$, and are isomorphisms for all $m\ge n$.
\end{definition}

Thus trivial hypergroupoids are hypercovers, but note that since the epimorphism condition means that we are only regarding a morphism in $\Aff^{nc}(R) $ as being surjective if it has a section, every trivial hypergroupoid is a  levelwise trivial fibration, and hence a levelwise  weak equivalence.

\begin{definition}\label{geomdef}
 Define an NC prestack $F$ to be strongly quasi-compact (sqc)  $n$-geometric    Artin (resp.  Deligne--Mumford) if it arises as the geometric realisation of an NC Artin (resp. NC Deligne--Mumford)  $n$-hypergroupoid $X$. Say that a morphism $G \to F$ of such prestacks is submersive (resp. \'etale) if it arises as the geometric realisation of a relative NC Artin (resp. NC Deligne--Mumford)  $n$-hypergroupoid $ Y \to X$ with $Y_0 \to X_0$ submersive (resp. \'etale).
\end{definition}

\begin{remarks}\label{geomrmks}
% Since we have taken affine building blocks in  
Definition \ref{geomdef} %rather than allowing infinite disjoint unions, it 
gives an analogue of the strongly quasi-compact  $n$-geometric Artin or Deligne--Mumford  stacks in \cite{hag2}. More precisely, if $X$ is an sqc $n$-geometric  Artin (resp. Deligne--Mumford) NC prestack, then restricting the functor to commutative $R$-algebras and taking \'etale hypersheafification yields a strongly quasi-compact   $n$-geometric  Artin (resp. Deligne--Mumford) stack, by \cite[Theorem \ref{stacks2-bigthm}]{stacks2}. That same theorem can be used to give an alternative inductive characterisation of sqc  $n$-geometric    Artin NC prestacks as prestacks $F$ for which there exists a surjective submersion $U \to F$ from an NC affine $U$ such that $U\by^h_FU$ is sqc $(n-1)$-geometric. 

Unravelling the definitions, if $F$ is  an sqc $n$-geometric Artin  NC prestack $F$, then $F$ is equivalent to a limit-preserving functor $X$  for which the map $X(A) \to X(B)$ of simplicial sets is a Kan fibration for every 
nilpotent surjection $A \to B$  in $\Alg(R)$.  If $C \to B$ is another morphism in  $\Alg(R)$, this means that the fibre product $X(A)\by_{X(B)}X(C)$ is a homotopy fibre product. %Since NC affine schemes preserve limits, 
It thus follows that the natural map
\[
 F(A\by_BC) \to F(A)\by_{F(B)}^hF(C)
\]
is a weak equivalence, so $F$ satisfies the NC analogue of the homogeneity property from \cite{drep} (cf. \S \ref{hgssn} below), which is closely related to Schlessinger's conditions.%, which is sometimes known as infinitesimal cohesiveness on one factor, by analogy with \cite{lurie}.

Moreover, the epimorphism conditions imply that for all $A \in \Alg(R)$, the simplicial set $F(A)$ is a Kan complex, with the higher isomorphism conditions implying that it is in fact a form of  $n$-groupoid. In particular, its homotopy groups vanish above degree $n$. 

Finally, since the partial matching maps are submersive, they are l.f.p., which amounts to saying that for any filtered system $\{B_i\}_i$ in $\Alg(R)$ with colimit $B$, the map 
\[
 \LLim_iX(B_i) \to X(B)
\]
is a fibration. In particular, if $C \to B$ is another morphism in  $\Alg(R)$, this means that the fibre product $ X(C)\by_{X(B)}\LLim_iX(B_i)$ is a homotopy fibre product, so the natural diagram
\[
\begin{CD}
 \LLim_i F( C\by_B B_i) @>>> F(C)\\
@VVV @VVV\\
\LLim_iF(B_i) @>>> F(B)
\end{CD}
\]
 is a homotopy pullback square.
\end{remarks}

\begin{definition}
 We say that a morphism $F \to G$ of sqc  $n$-geometric    Artin (resp.  Deligne--Mumford) NC prestacks is submersive (resp. \'etale) if it arises as the geometric realisation of a relative NC Artin (resp. NC Deligne--Mumford)  $n$-hypergroupoid $X \to Y$ for which the map $X_0 \to Y_0$ is submersive (resp. \'etale).
\end{definition}

\begin{definition}\label{opendef}
 We say that a morphism  $F \to G$ of sqc  $n$-geometric    Artin (resp.  Deligne--Mumford) NC prestacks is open if 
\begin{enumerate}
 \item it is submersive (resp. \'etale), and
\item it is a levelwise monomorphism in the sense that for all  $A \in \Alg(R)$, the  map $F(A) \to G(A)$ induces an injection on $\pi_0$ and an isomorphism on the homotopy groups $\pi_i$ for $i \ge 1$.
\end{enumerate}
\end{definition}

\begin{remark}\label{openrmk}
We may rephrase formal submersiveness  of a morphism $F \to G$ of $n$-geometric Artin NC prestacks as saying that for all nilpotent surjections $A \to B$ of $R$-algebras, the map
\[
\eta \co  F(A) \to G(A)\by^h_{G(B)}F(B)
\]
to the homotopy fibre product is surjective on $\pi_0$, since  \cite[Theorem \ref{stacks2-relstrict}]{stacks2} allows us to represent any morphism of $n$-geometric stacks by a relative $n$-hypergroupoid. For submersiveness of such a morphism, we need to add an l.f.p. condition, which amounts to saying that for any filtered system $\{C_i\}_i$ of $R$-algebras with colimit $C$, the map
\[
 \LLim_i F(C_i) \to F(C)\by^h_{G(C)}\LLim_i G(C_i)
\]
is an equivalence.

Similarly, \'etaleness  of a morphism $F \to G$ of  $n$-geometric     Deligne--Mumford NC prestacks amounts to saying that $\eta$ is a weak equivalence, together with the same l.f.p. condition. 

If $F \to G$ is a levelwise monomorphism, then %$F(A) \simeq G(A)\by_{\pi_0G(A)}\pi_0F(A)$ and $F(B) \simeq G(B)\by_{\pi_0G(B)}\pi_0F(B)$, 
$\eta$ reduces to the map
\[
 G(A)\by_{\pi_0G(A)}\pi_0F(A) \to G(A)\by_{\pi_0G(B)}\pi_0F(B),
\]
so submersiveness reduces to saying that $\pi_0F(A) \to \pi_0G(A)\by_{\pi_0G(B)}\pi_0F(B)$ is surjective, hence an isomorphism since both sides embed in $\pi_0G(A)$. \'Etaleness (where defined) also reduces to the same condition. 

In particular, this means that if we regard  $n$-geometric     Deligne--Mumford NC prestacks as being $n$-geometric     Artin NC prestacks, the two potential notions of openness from Definition \ref{opendef} agree.
\end{remark}

We are now in a position to be able to pass beyond the strongly quasi-compact setting:
\begin{definition}\label{inftygeomdef}
 Define an NC prestack $F$ to be $n$-geometric    Artin (resp.  Deligne--Mumford) if it arises as a filtered colimit of open morphisms of sqc $n$-geometric    Artin (resp.  Deligne--Mumford) NC prestacks (for fixed $n$). Define an NC prestack $F$ to be $\infty$-geometric    Artin (resp.  Deligne--Mumford) if it arises as a filtered colimit open morphisms of sqc $m$-geometric    Artin (resp.  Deligne--Mumford) NC prestacks, for varying $m$.  
\end{definition}

The following is now an immediate consequence of Lemma \ref{weillemma1}:
\begin{lemma}\label{stweillemma1}
If $S \in \Alg(R)$ is finite and flat as an $R$-module, then the  restriction of scalars functor $\prod_{S/R}$ on $ \Aff^{nc}(R)^{\wedge}$  preserves the subcategories of  $n$-geometric   and $\infty$-geometric  Artin and  Deligne--Mumford NC prestacks, with $\prod_{S/R}F$ being sqc whenever $F$ is so.
\end{lemma}

\subsubsection{Moduli of projective modules}\label{modprojmodsn}

\begin{definition}
 Given a non-unital associative $R$-algebra $A$, define the groupoid $\cI(A)$ of idempotents of $A$ as follows. 

The set $\Ob_{\cI}(A)$ of objects of $\cI(A)$ is defined to be 
 the set of idempotents
\[
 \{e \in A ~:~ e^2=e\}.
\]

The set $\Iso_{\cI}(A)$ of isomorphisms in $\cI(A)$ is defined to be 
the set of pairs
\[
 \{(f,g) \in A \by A~:~ fgf=f, ~ gfg=g\}.
\]
The source and target  of the isomorphism $(f,g)$ are defined to be the idempotents $gf$ and $fg$ respectively. The identity at $e \in \Ob_{\cI}(A)$ is given by $(e,e)$, and composition is given by multiplication, so   $(f',g') \circ (f,g):= (f'f,gg')$ whenever $fg= g'f'$. %%source $gg'f'f=gfgf=gf$, target $f'fgg'=f'g'f'g'=f'g'$.
Note that the inverse of $(f,g)$ is simply $(g,f)$.
\end{definition}

\begin{remarks}\label{idemreprmks}
Observe that the functor $\Ob_{\cI}$ is represented by $R$, regarded as a non-unital algebra, since a non-unital $R$-algebra homomorphism $R \to A$ is determined by the image of $1 \in R$, which must be idempotent.

The functor $\Iso_{\cI}$ is also representable: take  the free non-unital $R$-algebra on generators $f,g$, and quotient by the relations $fgf=f$ and $gfg=g$.

Indeed, for any non-unital associative $R$-algebra $B$ which is finite and projective as an $R$-module, a similar construction shows that the functors $\Ob_{\cI}(B\ten_R-) $ and $\Iso_{\cI}(B\ten_R-) $ are representable. In particular, we may apply this to matrix rings $B=\Mat_n(R)$, in which case $B\ten_RA \cong \Mat_n(A)$, and $\cI(\Mat_n(A))$ is a groupoid of idempotent matrices. 
\end{remarks}

\begin{lemma}\label{idemmodlemma}
The groupoid $\cI(A)$ is equivalent to the groupoid of those right $A$-modules which are direct summands of $A$, while the groupoid $\cI(\Mat_n(A))$ is equivalent to the groupoid of those projective right $A$-modules which  admit a set of $n$ generators. % are direct summands of $A^n$.
\end{lemma}
\begin{proof} 
We begin by defining a functor from $\cI(A)$ to the category of right $A$-modules.
On objects, this functor sends an idempotent $e$  to the right $A$-module $eA$. 
Note that if we have $(f,g) \in \Iso_{\cI}$ with source $e:=gf$ and target $e':=fg$, then $e'f=f$ and  $fe=f$. In particular, this means that left multiplication by $f$ defines a right $A$-module isomorphism from $eA$ to $e'A$, with inverse given by $g$.  It is then easy to see  that the resulting map
\[
 \Iso_{\cI}(e,e') \to \Hom_A(eA, e'A) 
\]
satisfies all the conditions required of  a functor.

Similarly, if $e \in \cI(\Mat_n(A))$, then $e(A^n)$ is a direct summand of the right $A$-module $A^n$,  or equivalently a projective $A$-module with $n$ generators,    and a pair  $(f,g) \in \Iso_{\cI}(e,e')$ is equivalent to an invertible element of 
\[
  \Hom_A(eA^n, e'A^n) \cong e'\Mat_n(A)e,
\]
giving rise to an equivalence from   $\cI(\Mat_n(A))$ to the groupoid of those right $A$-modules which are direct summands of $A^n$.
\end{proof}

Composing $\cI$  with the forgetful functor from unital to non-unital rings gives us a functor (also denoted $\cI$) from $\Alg(R)$ to groupoids, which we now study. 
\begin{lemma}\label{idsmoothlemma}
 The source and target maps $\Iso_{\cI} \to \Ob_{\cI}$ are both submersive maps of representable functors on $\Alg(R)$. The representable functor $\Ob_{\cI}$ is also submersive.
\end{lemma}
\begin{proof}
 If a functor on  non-unital rings is represented by a non-unital ring $A$, then the associated functor on unital rings is represented by the unital ring $R \oplus A$, so  $\Iso_{\cI} \to \Ob_{\cI}$ are both representable as functors on $\Alg(R)$. The representing algebras are finitely generated, because their non-unital precursors were also. It remains to establish formal submersiveness.

We begin by considering the source map. Given a surjection $\phi \co A \to B$ in $\Alg(R)$ with square-zero kernel $J$, we need to show that 
\[
 \Iso_{\cI}(A) \xra{(s,\phi)} \Ob_{\cI}(A)\by_{\Ob_{\cI}(B)}\Iso_{\cI}(B)
\]
is surjective. Take an element $\tilde{e} \in \Ob_{\cI}(A)$ with image $e \in \Ob_{\cI}(B)$, and an isomorphism $(f,g) \in \Iso_{\cI}(B)$ with source $e$, so $e=gf$. We need to show that $(f,g)$ lifts to an element of $\Iso_{\cI}(A)$ with source $\tilde{e}$.

First, use surjectivity of $\phi$ to lift $f$ to an element $\tilde{f} \in  A$. Since $fe=f$, we may post-multiply by $\tilde{e}$, and thus  assume that $\tilde{f}\tilde{e}=\tilde{f}$. Next,  lift $g$ to an element $g' \in  A$; similarly, we may assume that $g'= \tilde{e}g'$. Now let $x= g'\tilde{f} - \tilde{e}$; this is an element of $J$, since its image in $B$ is $gf-e=0$. Setting $\tilde{g}:= (1-x)g'$ gives an element with
\[
 \tilde{g}\tilde{f}= (1-x)g'\tilde{f}= (1-x)(\tilde{e}+x)= \tilde{e}-xe+x=\tilde{e}
\]
(noting that $x^2=0$), so $(\tilde{f}, \tilde{g})$ has source $\tilde{e}$. Finally, note that
\[
 \tilde{g}\tilde{f}\tilde{f}= \tilde{e}\tilde{f}=\tilde{f}, \quad \tilde{f}\tilde{g}\tilde{f}= \tilde{f}\tilde{e}=\tilde{f},
\]
so $(\tilde{f}, \tilde{g})$ is indeed an element of $\Iso_{\cI}(A)$.

Formal submersiveness of   the target map now follows by symmetry, and it only remains to show that $\Ob_{\cI}$ is submersive. Given an idempotent $e \in \cI(B)$, lift it to an element $y \in A$, and set $\tilde{e}:= 3y^2-2y^3$. This is another lift of $e$, and it is necessarily idempotent because $( 3y^2-2y^3)^2 \equiv (3y^2-2y^3) \mod (y^2-y)^2$ and $(y^2-y)\in J$.
\end{proof}

\begin{definition}
 Define $\Mat_{\infty}(A)$ to be the filtered colimit $\LLim_n \Mat_n(A)$ of the non-unital ring homomorphisms $\Mat_n(A)\to \Mat_{n+1}(A)$ given by extension by $0$. Thus an object of the non-unital ring $\Mat_{\infty}(A)$ can be regarded as an $\N \by \N$-matrix with only finitely many non-zero entries. 
\end{definition}

The following is an immediate consequence of Lemma \ref{idemmodlemma}
\begin{proposition}\label{idmatprop}
 The groupoid $\cI(\Mat_{\infty}A)$ is equivalent to the groupoid $\cP(A) $ of finite projective right $A$-modules.
\end{proposition}

\begin{proposition}\label{projgeomprop}
The simplicial functor on $\Alg(R)$ sending an associative algebra $A$ to the nerve $B\cP(A)$ is a submersive $1$-geometric Artin NC prestack.
\end{proposition}
\begin{proof}
By Proposition \ref{idmatprop}, the NC prestack $B\cP$ is equivalent to the filtered colimit $\LLim_n B\cI(\Mat_n(-))$ of prestacks. By Remarks \ref{idemreprmks}, these NC prestacks can be represented by simplicial NC affine schemes, and Lemma \ref{idsmoothlemma} implies that these simplicial NC affine schemes are submersive Artin $1$-hypergroupoids. 

It therefore remains only  to show that the maps $B\cI (\Mat_n(-))\to B\cI(\Mat_{n+1}(-))$ are open. That they are monomorphisms follows from Lemma \ref{idemmodlemma}. That the maps are submersive  amounts to saying that for any nilpotent surjection $A \to B$, if a  projective right $A$-module $M$ is such that the right $B$-module $M\ten_AB$ is a direct summand of $B^n$, then there is a compatible presentation of $M$ as a direct summand of $A^n$. To see this, first lift the $n$ induced generators of $M\ten_AB$ to elements of $M$; these will necessarily generate $M$ as an $A$-module. We then need a section $\tilde{\sigma} \co M \to A^n$ lifting the section $\sigma \co M\ten_AB \to B^n$; to construct this, start from the map $(\sigma, \id) \co M \to B^n\by_{M\ten_AB}M$, then lift to $\tilde{\sigma} \co M \to A^n$ using projectivity of $M$.
\end{proof}

\begin{remark}
 On restricting to commutative rings, the NC prestack $B\cP$ gives the geometric $1$-stack of projective modules. Note that geometricity of this prestack is not a consequence of Remarks \ref{geomrmks}, because we have not had to hypersheafify: the restriction of $B\cP$ to commutative rings is already an \'etale hypersheaf.   Proposition \ref{projgeomprop} is thus much more than an NC analogue of geometricity of the (commutative) stack of vector bundles, because we have a simplicial affine presentation of $B\cP$ which works without requiring any form of hypersheafification. 

By contrast, the standard presentation (requiring hypersheafification) of the stack of vector bundles over commutative rings is just $\coprod_n B\GL_n$, which amounts to taking the full subgroupoid of $B\cI(\Mat_{\infty})$ on diagonal idempotents.
\end{remark}

\begin{corollary}\label{projmodcor}
 Given a finite flat associative $R$-algebra $S$, the  simplicial functor on $\Alg(R)$ sending an associative algebra $A$ to the nerve $B\cP(S\ten_RA)$ is a submersive $1$-geometric Artin NC prestack.
\end{corollary}
\begin{proof}
This follows by combining Proposition \ref{projgeomprop} with Lemma \ref{stweillemma1}. 
%combined with the observation that finiteness and flatness of $S$ over $R$ imply that  for an NC affine scheme $X$, the functor $A \mapsto X(S\ten_RA)$ is another NC affine scheme (we can think of this as a non-commutative analogue of Weil restriction of scalars), and that this functor respects submersive morphisms, epimorphisms and open morphisms. 
\end{proof}

\begin{remark}
 Note that if we apply Corollary \ref{projmodcor} to the algebra $T_n(R)$ of upper triangular $n \by n$ matrices, then we establish $1$-geometricity of the prestack $B\cP( T_n\ten-)$ of flags of length $n$. 
\end{remark}

\subsubsection{Prestacks of dg categories}\label{dgcatsn2}

We now consider dg category-valued moduli functors, which are similar in spirit to the $(\infty,1)$-category-valued functors considered in \cite{balchinAugmentedHAG}.

The following adapts Definition \ref{npreldef} from simplicial sets to dg categories:

\begin{definition}\label{Artindgcatdef}
 Define an sqc NC Artin %(resp. NC Deligne--Mumford) 
 dg category over $R$ to be a functor $\C$ from $\Alg(R)$ to dg categories with the following properties:
\begin{enumerate}
 \item the presheaf $\Ob \C$ of objects of $\C$ is affine, i.e. lies in $\Aff^{nc}(R)$;
\item for each $n \in \Z$, the functors $A \mapsto \coprod_{(x,y) \in \Ob \C(A)} \C(x,y)_n$ are affine;
\item for any nilpotent surjection $A \to B$ in $\Alg(R)$, the morphism
\[
 \C(A) \to \C(B)
\]
is a fibration   of dg categories for the model structure of \cite{tabuadaMCdgcat}; %this says for all $x,y \in \C(A)$, the map $\C(A)(x,y) \to \C(B)(\bar{x},\bar{y})$ is a surjection of complexes, and that $\H_0\C(A)\to \H_0\C(B)$ is a fibration of categories, meaning that for any object $x \in \C(A)$ and any isomorphism $\gamma\co \bar{x} \to y$ in $\H_0\C(B)$, there is a lift of $\gamma$ to an isomorphism in $\H_0\C(A)$ based at $x$.

\item  
for any filtered system $\{A_i\}_i$ in $\Alg(R)$ with colimit $A$, the map
\[
 \LLim_i\C(A_i) \to \C(A)
\]
is a fibration of dg categories.
\end{enumerate}
Moreover, we say that $\C$ has $n$-truncated $\Hom$-complexes if $\C(x,y)_i=0$ for all $i> n$ (homological, not cohomological, grading) and all $x,y \in \Ob \C(A)$.
\end{definition}

\begin{remark}
 We may rephrase the third condition in Definition \ref{Artindgcatdef} as saying that for all  $x,y \in \C(A)$, the map $\C(A)(x,y) \to \C(B)(\bar{x},\bar{y})$ is a surjection of complexes, and that $\H_0\C(A)\to \H_0\C(B)$ is a fibration of categories, meaning that for any object $x \in \C(A)$ and any isomorphism $\gamma\co \bar{x} \to y$ in $\H_0\C(B)$, there is a lift of $\gamma$ to an isomorphism in $\H_0\C(A)$ with source $x$.

We may similarly rephrase the fourth condition as saying that for all  $x,y \in \LLim_i\C(A_i)$, the map $\LLim_i\C(A_i)(x,y) \to \C(A)(\bar{x},\bar{y})$ is a surjection of complexes, and that $ \LLim_i\H_0\C(A_i)\to \H_0\C(A)$ is a fibration of categories.

The reason we do not have an analogue of the epimorphism condition from Definition \ref{npreldef} is that every dg category is fibrant, so it is redundant when considering NC Artin dg categories $\C \to \Spec R$.
\end{remark}

As for instance in \cite{dmsch}, for any  dg category-valued moduli functor, there is an associated simplicial set-valued functor describing moduli of objects, constructed as follows. 

 We first form a dg category $\tau_{\ge 0}\C$ by taking good truncation of the  $\Hom$-complexes in non-negative homological degrees, so
\[
 (\tau_{\ge 0}\C)(x,y)_i:= \begin{cases} \C(x,y)_i & i>0 \\ \z_0\C(x,y) & i=0\\ 0 & i<0. \end{cases} 
\]
We can now apply the Dold--Kan denormalisation functor $N^{-1}$ (denoted by $K$ in \cite[8.4.4]{W}) to each $\Hom$-complex, giving simplicial sets $N^{-1}(\tau_{\ge 0}\C)(x,y)$. Applying the Alexander--Whitney map (cf. \cite[8.5.4]{W}) to composition morphisms in $\tau_{\ge 0}\C$ then gives  composition maps for $N^{-1}(\tau_{\ge 0}\C)(x,y)$, making it into a simplicial category.

Given the simplicial category $N^{-1}\tau_{\ge 0}\C$, we can then form the subcategory $\cW(N^{-1}\tau_{\ge 0}\C)$ of homotopy equivalences, so $\cW(N^{-1}\tau_{\ge 0}\C)(x,y) \subset N^{-1}\tau_{\ge 0}\C(x,y)$ is the union of path components over invertible morphisms in $\H_0\C(x,y)$. This simplicial category can be regarded as a form of $\infty$-groupoid, and we can then take the nerve to give a simplicial set $\bar{W}\cW(N^{-1}\tau_{\ge 0}\C)$. By comparing universal properties and homotopy groups, an alternative characterisation is as the derived mapping space
\[
 \bar{W}\cW(N^{-1}\tau_{\ge 0}\C) \simeq \map_{dg\Cat}(\Z, \C) 
\]
 from the dg category $\Z$, consisting of  one object $*$ with $\End(*)=\Z$, or as 
the derived mapping space
\[
\map_{s\Cat}(*, N^{-1}\tau_{\ge 0}\C)
\]
in simplicial categories.

\begin{proposition}\label{coreprop}
 If $\C$ is an sqc NC Artin %(resp. NC Deligne--Mumford) 
 dg category over $R$ with $n$-truncated $\Hom$-complexes, then the simplicial presheaf 
\[
 A \mapsto \map_{dg\Cat}(\Z, \C(A))
\]
is an sqc  $(n+1)$-geometric NC Artin prestack.
\end{proposition}
The proof will proceed by calculating  $\map_{dg\Cat}(\Z, \C(A))$ as a  derived left function complex, so we first need a  cosimplicial resolution of the dg category $\Z$.
\begin{lemma}
 There exists a cosimplicial  resolution $J^{\bt}$ of $\Z$ in $dg\Cat$ which is cofibrant for the Reedy model structure, and such that for $r \ge 1$ the cosimplicial latching maps $L^rJ \to J^r$ (in the notation of \cite[\S 5.2]{Hovey}) are: 
\begin{enumerate}
 \item isomorphisms on objects, and
\item given on the underlying graded categories by freely adding finitely many homomorphisms in chain degrees $\ge n-1$. 
\end{enumerate}
\end{lemma}
\begin{proof}[Proof of lemma (sketch).]
We may set $J^0$ to be the one-object category $\Z$, which is cofibrant, and then set $J^1$ to be the dg category $\cK$ of \cite{tabuadaMCdgcat} which is quasi-equivalent to $\Z$ and has two objects, with the map $\Z \coprod \Z \to J^1$ being a cofibration freely generated by five homomorphisms of various degrees. 

We now proceed inductively, so assume that we have constructed the cosimplicial resolution $J^{\bt}$ to level $r-1$, and thus need a factorisation $L^rJ \to J^r \to M^rJ$ of the map $L^rJ \to M^rJ$ from the $r$th latching object to the $r$th matching object, such that $L^rJ \to J^r$ is a cofibration and the map $J^r \to \Z$ is a quasi-equivalence. 

Now, for any dg category $\cB$, taking iterated colimits gives
\[
\map_{dg\Cat}(L^rJ,\cB)\simeq \map_{s\Set}(\pd \Delta^r,\map_{dg\Cat}(\Z,\cB)),
\]
and thus the homotopy fibre $F_x$ of $\map(L^rJ,\cB)\to \map(\Z,\cB) $ over an object $x \in \cB$ is given by the Dold--Kan denormalisation 
$
 N^{-1}\tau_{\ge 0}(\cB(x,x)_{[r-2]})
$
for $r \ge 3$, while for $r=2$ we have to restrict to quasi-isomorphisms, giving the homotopy fibre as 
$
 (N^{-1}\cB(x,x))\by_{\H_0\cB(x,x)}\H_0\cB(x,x)^{\by}
$.

For an object $j \in J$, the identity functor on $J$ thus gives us an element of $\H_{r-2}(L^rJ)(j,j)$ which is invertible when $r=2$. Take a representative $a \in \z_{r-2}(L^rJ)(j,j)$ and form $J^r$ by freely attaching a morphism in chain degree $r-1$ with $\delta h =a$, or $\delta h=a-1$ when $r=2$. Now look at the homotopy fibre $G_x$ of $\map(J^r,\cB) \to   \map(\Z,\cB)$ over $x$. By construction, the map $G_x \to F_x$ is homotopic to a constant and has homotopy fibres  $N^{-1}\tau_{\ge 0}(\cB(x,x)_{[r-1]})$, from which it follows that $G_x$ must be contractible.

We have therefore constructed a latching map $L^rJ \to J^r$ by freely adding a homomorphism in chain degree $r-1$, with the map $J^r \to \Z$ being a quasi-equivalence. It remains only to show that we have a compatible map $J^r \to M^rJ$, but this follows because $M^rJ \to \Z$ is a quasi-equivalence (seen for instance by analysing the cosimplicial Dold--Kan correspondence), so the image of $a$ in $M^rJ$ is homologically trivial, and we may map $h$ to a suitable homotopy.
\end{proof}

\begin{proof}[Proof of Proposition \ref{coreprop}.]
 Since every dg category is fibrant in the model structure of \cite{tabuadaMCdgcat}, we may apply \cite[Theorem 5.4.9]{Hovey}, calculating $\map$ as the derived left function complex  $\oR\Map_l$. In other words,  a model for $\map_{dg\Cat}(\Z, \C(A))$ is given by the simplicial set
\[
i \mapsto  \Hom_{dg\Cat}(J^i, \C(A)),
\]
for the cosimplicial cofibrant resolution $J^{\bt}$ of $\Z$ in $dg\Cat$ given by the lemma.

By Kan extension, $J^{\bt}$ gives rise to a colimit-preserving functor $J \co s\Set \to dg\Cat$ with $J(\Delta^r)=J^r$; in particular, $L^rJ=J(\pd \Delta^r)$. Once we observe that the inclusion $\L^r_k \into \pd \Delta^r$ of a horn in the boundary of an $r$-simplex is a pushout of the map $\pd\Delta^{r-1} \to \Delta^{r-1}$, it follows from the lemma that the partial latching maps $J(\L^r_k) \to J^r$ are cofibrations freely generated by finitely many homomorphisms in chain degrees $\ge r-2$, for $r \ge 2$. In particular, the morphisms $\sigma_{\le r-3}J(\L^r) \to \sigma_{\le r-3}J^r$ of brutal truncations are isomorphisms for $r \ge 3$.

Since $J^{\bt}$ is a resolution of $\Z=J^0$, we also know that the coface maps $\pd^i \co J^{r-1} \to J^r$ are all quasi-equivalences. We may use this to deduce that the partial latching maps $J(\L^r_k) \to J^r$ are all trivial cofibrations, as follows. For $r=1$, the partial latching maps are already coface maps. If the statement holds for all $1 \le r<m$, then $J(K) \to J(L)$ is a trivial cofibration for all trivial cofibrations $K \to L$ of simplicial sets generated in degrees $<m$. In particular, each face of $\L^m_k$ induces a trivial cofibration $J^{m-1} \to  J(\L^m_k)$ whose composition with $J(\L^m_k) \to J^m$  is a coface map $J^{m-1} \to J^m$, and hence a quasi-equivalence. By the two-out-of-three property, we deduce that the partial latching map  $J(\L^m_k) \to J^m$ is a quasi-equivalence, and hence a trivial cofibration.

For any fibration $\cB \to \cD$ of dg categories, we may therefore deduce that the map
\[
 \Hom_{dg\Cat}(J^{\bt}, \cB) \to \Hom_{dg\Cat}(J^{\bt}, \cD)
\]
of simplicial sets is a Kan fibration, since $\Hom_{s\Set}(\L^r_k,  \Hom_{dg\Cat}(J^{\bt}, \cB))=  \Hom_{dg\Cat}(J(\L^r_k), \cB)$.

The isomorphisms $\sigma_{\le r-3}J(\L^r) \to \sigma_{\le r-3}J^r$   above then imply that the good truncations $\tau_{\le r-2}J(\L^r_k) \to \tau_{\le r-2}J^r$ are all isomorphisms.  For any fibration $\cB \to \cD$ of dg categories with $n$-truncated $\Hom$-complexes, it follows that  the Kan fibrations
\[
 \Hom_{dg\Cat}(J^{\bt}, \cB) \to \Hom_{dg\Cat}(J^{\bt}, \cD)
\]
are in fact relative $(n+1)$-hypergroupoids, since the partial matching maps are isomorphisms in degrees $\ge n+2$. 

Since $\C$ is  an sqc NC Artin %(resp. NC Deligne--Mumford) 
 dg category and every dg category is fibrant,  this implies that the partial matching maps of the simplicial NC presheaf  $ \Hom_{dg\Cat}(J^{\bt},\C(-))$ are all formally submersive epimorphisms, and are isomorphisms in degrees $\ge n+2$. Since the $J(\L^m_k) \to J^m$ are all finitely generated, it follows   that the partial matching maps of   $\Hom_{dg\Cat}(J^{\bt},\C(-))$ are all l.f.p., and hence submersive. Thus we have shown that the model $\Hom_{dg\Cat}(J^{\bt},\C(-))$ for  $\map_{dg\Cat}(\Z, \C(-))$ is an
 an NC Artin  $(n+1)$-hypergroupoid, as required.
\end{proof}

\begin{definition}\label{opendefcat} 
 We say that a morphism  $F \to G$ of sqc   NC  Artin dg categories over $R$ is
\begin{enumerate}
 \item formally submersive if for all for all nilpotent surjections $A \to B$ of $R$-algebras,
 the dg functor
\[
\eta \co  F(A) \to G(A)\by^h_{G(B)}F(B)
\]
to the homotopy fibre product is  surjective on homotopy classes of objects;

\item l.f.p. if for any filtered system $\{C_i\}_i$ of $R$-algebras with colimit $C$, the dg functor
\[
 \LLim_i F(C_i) \to F(C)\by^h_{G(C)}\LLim_i G(C_i)
\]
is a quasi-equivalence;
 
 \item submersive if it is formally submersive and l.f.p.;
 
 \item open if  it is submersive  and
 a levelwise monomorphism in the sense that for all  $A \in \Alg(R)$, the  map $F(A) \to G(A)$ induces an injection on homotopy classes of objects  and  quasi-isomorphisms  on $\Hom$-complexes.
\end{enumerate}
\end{definition}

\subsubsection{Moduli of perfect complexes}\label{modperfsn}

\begin{definition}
 Given a DGAA $A$, write $\per_{dg}(A)$ for the dg category of cofibrant  perfect complexes of right $A$-modules; for alternative characterisations, see \cite[\S 4.6]{kellerModelDGCat}.
\end{definition}

\begin{definition}\label{cPidef}
For $A \in \Alg(R)$, define a dg category $\cP_{\uline{i}}(A)$ as follows. We let the objects of $\cP_{\uline{i}}(A)$  consist of pairs $(e, \delta)$, where $e \in \prod_{n \in \Z} \Mat_{i_n}(A)$ is idempotent and $\delta \in \prod_{n \in \Z}\Mat_{i_n,i_{n-1}}$ satisfying 
\begin{align*}
 0&= \delta_n \circ \delta_{n+1} \in \Mat_{i_{n+1},i_{n-1}},\\
 \delta_n &= e_{n-1} \circ \delta \circ e_n.
\end{align*}
Thus $\delta$ gives precisely the data to turn the graded projective  right $A$-module $\bigoplus_n (e_nA^{i_n})_{[-n]}$ into a chain complex $P(e,\delta)$ of right $A$-modules, and hence a perfect complex. 

We then define the morphisms in $\cP_{\uline{i}}(A)$ by setting
\[
 \cP_{\uline{i}}(A)((e,\delta),(e',\delta')):=\HHom_A(P(e,\delta),P(e',\delta'));
\]
explicitly,
\[
 \cP_{\uline{i}}(A)((e,\delta),(e',\delta'))_r:= \prod_{n \in \Z} \{ f \in \Mat_{i_n,i_{n+r}}(A) ~:~ f_n= e_{n+r}\circ f_n \circ e_n\},
\]
with differential $\delta(f):= \delta \circ f \mp f \circ \delta$.  
\end{definition}
This set-valued functor of objects of $\cP_{\uline{i}}$ is clearly representable by a finitely generated $R$-algebra,  with generators corresponding to the co-ordinates of $(e,\delta)$ and relations corresponding to the equations they must satisfy. Similarly, each functor of morphisms of fixed degree in $\cP_{\uline{i}}$ is clearly representable by a finitely generated $R$-algebra.

\begin{proposition}\label{cPiArtinprop}
 The functors $\cP_{\uline{i}} $ are  sqc NC Artin 
 dg categories with truncated $\Hom$-complexes.
\end{proposition}
\begin{proof}
 We have already seen that the object and morphism functors associated to $\cP_{\uline{i}} $ are representable  by finitely generated $R$-algebras. They are therefore affine, and for any filtered system $A_j$ of $R$-algebras, the map $\LLim_j\cP_{\uline{i}}(A_j) \to \cP_{\uline{i}}(\LLim_j A_j)$ is an isomorphism, hence a fibration. We know that the values of $i_n$ are zero outside some interval $[M,N]$, which implies that $\cP_{\uline{i}} $ has $(N-M)$-truncated $\Hom$-complexes.
 
Thus the only remaining condition to verify from Definition \ref{Artindgcatdef} is that for any 
 nilpotent surjection $\phi \co A \to B$ in $\Alg(R)$ with square-zero kernel $I$, we need  the morphism $\cP_{\uline{i}}(A)\to \cP_{\uline{i}}(B)$ to be  a fibration of dg categories. 
 
 Given objects $(e,\delta),(e',\delta') \in \cP_{\uline{i}}(A)$, the morphism
\begin{align*}
 \cP_{\uline{i}}(A)((e,\delta),(e',\delta'))_r&\to \cP_{\uline{i}}(B)(\phi(e,\delta),\phi(e',\delta'))_r\\
 \prod_{n \in \Z} e_{n+r}\Mat_{i_n,i_{n+r}}(A)e_n &\to \prod_{n \in \Z}\phi(e_{n+r})\Mat_{i_n,i_{n+r}}(B)\phi(e_n)
\end{align*}
is clearly surjective from the surjectivity of $A \to B$. 

It remains to show that given $(e,\delta) \in \cP_{\uline{i}}(A)$, every isomorphism $\theta \co (e,\delta) \to (f,\eps)$ in $\H_0\cP_{\uline{i}}(B)$ lifts to an isomorphism $\tilde{\theta}$ in  $\cP_{\uline{i}}(A)$ with source $(e,\delta)$. 
By Lemma \ref{idsmoothlemma}, we may lift $f$ to an idempotent $\tilde{f}$ in $\prod_{n \in \Z} \Mat_{i_n,i_{n+r}}(A)$. Now choose any lifts of $\theta$ and $\eps$ to elements  $\theta' \in \prod_{n \in \Z} f_n\Mat_{i_n}(A)e_n $ and $\eps' \in \prod_{n \in \Z} f_{n-1}\Mat_{i_n,i_{n-1}}(A)e_n $ respectively, and consider the obstructions to $\theta' \co (e,\delta) \to (\tilde{f},\eps')$ being a closed morphism in 
 $\cP_{\uline{i}}(A)$. Standard arguments show that the obstruction $( \eps' \circ \eps',  \eps' \circ \theta'-\theta' \circ \delta)$ lies in
\[
 \z^{1}\cone(\HHom_B(P(f,\eps), P(f,\eps)\ten_BI) \xra{-\circ \theta}\HHom_B(P(e,\delta), P(f,\eps)\ten_BI));
\]
since $\theta$ is a quasi-isomorphism, this cone complex is acyclic, so that cycle is the boundary of an element $(a,b)$, and then setting $(\tilde{\eps}, \tilde{\theta}):= (\eps'-a, \theta'-b)$ gives a lift for which the obstruction vanishes. It then remains only to observe that the morphism $\tilde{\theta} \co P(e,\delta)\to P(\tilde{f}, \tilde{\eps})$ is a quasi-isomorphism since the maps
$\tilde{\theta}I= \theta I \co P(e,\delta)I\to P(\tilde{f}, \tilde{\eps})I $ and $ \theta \co P(e,\delta)/P(e,\delta)I\to P(\tilde{f}, \tilde{\eps})/P(\tilde{f}, \tilde{\eps})I$ are both quasi-isomorphisms.
% \begin{align*}
%   \eps' \circ \eps'&=0\in  \z^{2}\HHom_B(P(f,\eps), P(f,\eps)\ten_BI)
%   \eps' \circ \theta'-\theta' \circ \delta=0 \in \HHom_B(P(e,\delta), P(f,\eps)\ten_BI)_{-1},
%  \end{align*}
% with the ideal $I$ appearing because the images of these are $0$ in $\cP_{\uline{i}}(B)$, while $\eps'\circ \eps'$ is closed because 
% \[
% 0= [\eps',\eps'\circ \eps'] = [\eps, \eps'\circ \eps'].
% \] 
% These also satisfy the condition
% \[
%  (\eps' \circ \eps') \circ \theta  = \eps' \circ (\eps' \circ \theta'-\theta' \circ \delta) + (\eps' \circ \theta'-\theta' \circ \delta) \circ \delta,
% \]
% i.e. $u \circ \theta$.
\end{proof}

\begin{corollary}\label{perdgArtin}
The NC presheaf $\per_{dg}\co \Alg(R) \to dg\Cat$  is quasi-equivalent to a filtered colimit
\[
 \LLim_{\uline{i} \in I} \cP_{\uline{i}}
\]
of  sqc NC Artin 
 dg categories with truncated $\Hom$-complexes, with all the transition maps being open in the sense of Definition \ref{opendefcat}.
\end{corollary}
\begin{proof}
For $A \in \Alg(R)$, the dg category $\per_{dg}(A)$ is quasi-equivalent to the dg category of bounded chain complexes of  finite projective right $A$-modules. We let $I \subset \N_0^{\Z}$ consist of sequences with only finitely many non-zero elements, with the partial order $\uline{i} \le \uline{j}$ if and only if $i_n\le j_n$ for all $n \in \Z$. We can then set $\per_{dg,\uline{i}}(A)$ to be the full dg subcategory of $\per_{dg}(A)$ consisting of complexes $M$ of projective $A$-modules where each $M_n$ admits a set of $i_n$ generators, so 
\[
 \per_{dg}(A) \simeq \LLim_{\uline{i} \in I}\per_{dg,\uline{i}}(A).
\]

The dg functors $\cP_{\uline{i}}\to\per_{dg,\uline{i}} $ given by $(e, \delta) \mapsto P(e, \delta)$ are clearly quasi-equivalences, Proposition \ref{cPiArtinprop} provides the required filtered system of sqc NC Artin 
 dg categories with truncated $\Hom$-complexes.

It only remains to show that  the maps $\per_{dg,\uline{i}}\to \per_{dg,\uline{j}}$ are open. They are monomorphisms by definition, and we have already seen that the prestacks $\cP_{\uline{i}}$ are l.f.p., so the only remaining condition to verify from Definition \ref{opendefcat} is that the transition maps  $\per_{dg,\uline{i}}\to \per_{dg,\uline{j}}$ are formally submersive. Given a nilpotent surjection $A \to B$ in $\Alg(R)$ with kernel $I$, we may argue as in the proof of Proposition \ref{cPiArtinprop} to see that the obstruction to lifting a complex $P$ from   $\per_{dg,\uline{i}}(B)$ to $\per_{dg,\uline{i}}(A) $ lies in $\Ext^2_B(P,P\ten_BI)$, and that if this vanishes the homotopy classes of lifts are a torsor for $\Ext^1_B(P,P\ten_BI) $; since these do not depend on $\uline{i}$, the transition maps are all formally submersive.
\end{proof}

\begin{definition}\label{Perfdef}
 We define the NC prestack $\Perf\co \Alg(R) \to s\Set$ to be the nerve of the core of the simplicial category  associated to $\per_{dg}$. In the notation of \S \ref{dgcatsn2},   $\Perf(A):= \bar{W}\cW(N^{-1}\tau_{\ge 0}\per_{dg}(A))$, and this is a model for $ \map_{dg\Cat}(\Z, \C(A))$. 
\end{definition}

\begin{corollary}\label{Perfgeomcor}
The NC prestack $\Perf$ is  Artin $\infty$-geometric.
\end{corollary}
\begin{proof}
By Corollary \ref{perdgArtin}, we have 
\[
 \Perf \simeq \LLim_{i \in I} \map_{dg\Cat}(\Z,\cP_i),
\]
and Proposition \ref{coreprop} implies that this 
a filtered colimit of open morphisms of Artin $n$-geometric NC prestacks, for varying $n$, as required.
\end{proof}

\subsection{Non-commutative derived $n$-geometric prestacks} \label{derivedgeomprestacksn}

\subsubsection{Definitions}\label{deriveddefsn}

We now  extend the constructions of the previous section to form non-commutative enhancements of derived moduli stacks, starting from our non-commutative derived affine schemes as the building blocks.

\begin{definition}\label{rKanh}\label{mnh}
For a simplicial object $X$ in a model category  $\C$, write  $\oR\Hom_{s\Set}(-,X)\co s\Set^{\op} \to \C$ for the  homotopy right Kan extension of $X$. This  preserves homotopy limits instead of limits, and can be realised as $\Hom_{s\Set}(-, \oR X)$, where $\oR X$ is a Reedy fibrant replacement for $X$ in $s\C$. 
\end{definition}

\begin{definition}%%restored this, as argt for apassing to all inputs is subtle, and want to avoid $\pi_0\oR\Lim$.
 Say that a morphism $X \to Y$ in $DG^+\Aff^{nc}(R)$ is a $\pi^0$-epimorphism if the induced map $\pi^0X \to \pi^0Y$ in $\Aff^{nc}(\H_0R)$ has a section.
\end{definition}

The following definition arises by substituting our context in \cite[Definitions \ref{stacks2-npreldef} and \ref{stacks2-nptreldef}]{stacks2}. %Note that an epimorphism $F \to G$ of prestacks is a morphism for which the maps $\pi_0F(A) \to \pi_0G(A)$ are surjective for all $A$. 
\begin{definition}\label{npreldef2}
Given  a simplicial diagram  $Y\in DG^+\Aff^{nc}(R)^{\Delta^{\op}}$, define a homotopy derived NC Artin (resp. NC Deligne--Mumford) $n$-hypergroupoid  over $Y$ to be a morphism $X_{\bt}\to Y_{\bt}$ in $ DG^+\Aff^{nc}(R)^{\Delta^{\op}}$, for which  the homotopy  partial matching maps
\[
X_m \to \oR\HHom_{s\Set}(\L^m_k, X)\by_{ \oR\HHom_{s\Set}(\L^m_k, X) }^hY_m 
\]
are  homotopy  submersive (resp. homotopy \'etale) $\pi^0$- epimorphisms for all  $k\ge 0,m\ge 1$, and are weak equivalences for all $m>n$ and all $k\ge 0$. When $Y= \Spec R$ is the final object, we will simply refer to $X$ as a derived homotopy NC Artin (resp. NC Deligne--Mumford)  $n$-hypergroupoid.
% 
% Define a  trivial derived homotopy  NC Artin (resp. NC Deligne--Mumford) $n$-hypergroupoid  over $Y$ to be a morphism $X_{\bt}\to Y_{\bt}$ in $ DG^+Aff^{nc}(R)^{\Delta^{\op}}$, for which 
% the homotopy  matching maps
% \[
% X_m \to  \oR\HHom_{s\Set}(\Delta^m, X)\by_{ \oR\HHom_{s\Set}(\Delta^m, X) }^hY_m 
% \]
% are  homotopy submersive (resp. homotopy \'etale) $\pi^0$-
% epimorphisms 
% for all $m \ge 0$, and are weak equivalences for all $m\ge n$.
\end{definition}

\begin{definition}\label{geomdef2}
 Define a derived NC prestack $F$ to be  strongly quasi-compact (sqc) $n$-geometric   NC Artin (resp. NC Deligne--Mumford) if it arises as the homotopy colimit of a derived homotopy NC Artin (resp. NC Deligne--Mumford)  $n$-hypergroupoid $X$. Say that a morphism $G \to F$ of such prestacks is submersive (resp. \'etale) if it arises as the geometric realisation of a relative homotopy NC Artin (resp. NC Deligne--Mumford)  $n$-hypergroupoid $ Y \to X$ with $Y_0 \to X_0$ homotopy submersive (resp. \'etale).
\end{definition}

\begin{remarks}\label{geomrmks2}
% Since we have taken affine building blocks in  
Definition \ref{geomdef2} %rather than allowing infinite disjoint unions, it 
gives an analogue of the strongly quasi-compact  $n$-geometric derived Artin or Deligne--Mumford  stacks in \cite{hag2}. More precisely, if $X$ is an sqc $n$-geometric  Artin (resp. Deligne--Mumford) derived NC prestack, then restricting the functor to CDGAs and taking \'etale hypersheafification yields a strongly quasi-compact   $n$-geometric  Artin (resp. Deligne--Mumford) derived stack, by \cite[Theorem \ref{stacks2-bigthm}]{stacks2}. 

Unravelling the definitions, if $A \to B$ is a nilpotent surjection in $dg_+\Alg(R)$ and $F$ is  an sqc derived $n$-geometric Artin  NC prestack $F$, then $F$ is equivalent to a Reedy fibrant limit-preserving functor $X$  for which the map $X(A) \to X(B)$ of simplicial sets is a Kan fibration. If $C \to B$ is another morphism in  $dg_+\Alg(R)$, this means that the fibre product $X(A)\by_{X(B)}X(C)$ is a homotopy fibre product. %Since derived NC affine schemes preserve limits, 
It thus follows that the map
\[
 F(A\by_BC) \to F(A)\by_{F(B)}^hF(C)
\]
is a weak equivalence, so $F$ satisfies the NC analogue of the homogeneity property from \cite{drep} (cf. \S \ref{hgssn} below).%, which is sometimes known as infinitesimal cohesiveness on one factor, by analogy with \cite{lurie}.

Similarly to Remarks \ref{geomrmks},  the homotopy l.f.p. condition for submersive morphisms then amounts to saying that for any filtered system $\{B_i\}_i$ in $dg_+\Alg(R)$ with colimit $B$, and any morphism $C \to B$, the natural map 
\[
 \LLim_i F( C\by_B^h B_i) \to F(C)\by_{F(B)}^h\LLim_iF(B_i)
\]
 is a weak equivalence.

Moreover, the weak equivalence conditions imply that for $F$ an  sqc derived  $n$-geometric Artin NC prestack and any $A \in \Alg(\H_0R)$, the homotopy groups  of $F(A)$ vanish above degree $n$. Beware, however, that the same will not be true for arbitrary $A \in \dg_+\Alg(R)$, although repeated application of the homogeneity property above to the Postnikov tower $\{A/\tau_{>n}A\}_n$  implies that if the homotopy groups of $A$ vanish above degree $m$, then those of $F(A)$ vanish above degree $m+n$. %%note that a homootpy submersive morphism is a weak equivalence iff it is a $\pi_0$-isomorphism.
The epimorphism conditions amount to saying that if $\oR F$ is a Reedy fibrant replacement for $F$, then $F(A)$ is a Kan complex for all $ A\in \Alg(\H_0R)$. 

Finally, if $f \co X \to Y$  is a trivial derived homotopy  NC Artin $n$-hypergroupoid, with $Y$ fibrant and $f$ a Reedy fibration, then the epimorphism conditions imply that $X(A) \to Y(A)$ is a trivial fibration for all $A \in \Alg(\H_0R)$, while the submersiveness conditions implies that for all nilpotent surjections $A \to B$, the map $X(A)\to Y(A)\by_{Y(B)}X(B)$ is a trivial fibration, with the target being a model for the homotopy fibre product. Combining these properties and applying them to the Postnikov tower, it follows that the map $X(A) \to Y(A)$ is a weak equivalence whenever the homology groups of $A$ are bounded; passing to homotopy limits, if follows that $X(A) \xra{\sim} Y(A)$ for all $A \in dg_+\Alg(R)$. Thus   trivial derived homotopy  NC Artin $n$-hypergroupoids become weak equivalences of the associated prestacks.
\end{remarks}

\begin{definition}
 We say that a morphism $F \to G$ of sqc derived $n$-geometric    Artin (resp.  Deligne--Mumford) NC prestacks is submersive (resp. \'etale) if it arises as the geometric realisation of a relative derived homotopy NC Artin (resp. NC Deligne--Mumford)  $n$-hypergroupoid $X \to Y$ for which the map $X_0 \to Y_0$ is submersive (resp. \'etale).
\end{definition}

\begin{definition}\label{opendef2}
 We say that a morphism  $F \to G$ of sqc derived  $n$-geometric    Artin (resp.  Deligne--Mumford) NC prestacks is open if 
\begin{enumerate}
 \item it is submersive (resp. \'etale), and
\item it is a levelwise monomorphism in the sense that for all  $A \in \Alg(R)$, the  map $F(A) \to G(A)$ induces an injection on $\pi_0$ and an isomorphism on the homotopy groups $\pi_i$ for $i \ge 1$.
\end{enumerate}
\end{definition}

\begin{remark}\label{openrmk2}
With the same reasoning as Remark \ref{openrmk}, 
 we may rephrase submersiveness (resp. \'etaleness) of a morphism $F \to G$ of sqc derived $n$-geometric Artin (resp. Deligne--Mumford) NC prestacks as saying that for all nilpotent surjections $A \to B$ of $R$-algebras, the map
\[
\eta \co  F(A) \to G(A)\by^h_{G(B)}F(B)
\]
to the homotopy fibre product is surjective on $\pi_0$ (resp. a weak equivalence).

Similarly, \'etaleness  of a morphism $F \to G$ of  $n$-geometric     Deligne--Mumford NC prestacks amounts to saying that $\eta$ is a weak equivalence. 

If $F \to G$ is a levelwise monomorphism, then 
both reduce to saying that $\pi_0F(A) \to \pi_0G(A)\by_{\pi_0G(B)}\pi_0F(B)$ is  an isomorphism. % since both sides embed in $\pi_0G(A)$. 
In particular, this means that if we regard derived  $n$-geometric     Deligne--Mumford NC prestacks as being derived $n$-geometric     Artin NC prestacks, the two potential notions of openness from Definition \ref{opendef2} agree.
\end{remark}

We are now in a position to be able to pass beyond the strongly quasi-compact setting:
\begin{definition}\label{inftygeomdef2}
 Define a derived  NC prestack $F$ to be $n$-geometric    Artin (resp.  Deligne--Mumford) if it arises as a filtered colimit of open morphisms of  sqc derived $n$-geometric    Artin (resp.  Deligne--Mumford) NC prestacks (for fixed $n$). Define a derived  NC prestack $F$ to be $\infty$-geometric    Artin (resp.  Deligne--Mumford) if it arises as a filtered colimit of sqc derived  $m$-geometric    Artin (resp.  Deligne--Mumford) NC prestacks, for varying $m$.  
\end{definition}

The following is now an immediate consequence of Lemma \ref{weillemma2}:
\begin{lemma}\label{stweillemma2}
If $S\in  \Alg(R)$ is finite and flat as an $R$-module, then the  restriction of scalars functor $\prod_{S/R}$ on $ DG^+\Aff^{nc}(R)^{\wedge}$  preserves the subcategories of  $n$-geometric   and $\infty$-geometric  Artin and  Deligne--Mumford derived NC prestacks, with $\prod_{S/R}F$ being sqc whenever $F$ is so.
\end{lemma}

\subsubsection{Homogeneity and tangent complexes}\label{hgssn}

As a prelude to derived representability results, we now adapt some definitions from \cite[\S 1.2]{drep}.

\begin{definition}
Say that a functor between model categories is homotopy-preserving if it maps weak equivalences to weak equivalences. 
\end{definition}
In particular, note that the fibrant derived NC prestacks $F:dg_+\Alg(R) \to s\Set$ are precisely those which are homotopy-preserving and objectwise fibrant.

\begin{definition}\label{sq0def}
We say that  a map $A \to B$ in $dg_+\Alg(R)$ is a square-zero extension if it is surjective and the kernel $I$ is square-zero, i.e. satisfies $I^2=0$. 
\end{definition}

\begin{definition}\label{hhgsdef}
We say that a %%homotopy-preserving 
functor 
$$
F:dg_+\Alg(R) \to s\Set
$$
is homogeneous if  for all square-zero extensions $A \to B$ and all maps $C \to B$ in $dg_+\Alg(R)$, the natural map
$$
F(A\by_BC) \to F(A)\by^h_{F(B)}F(C)
$$
to the homotopy fibre product is a weak  equivalence.
\end{definition}

This terminology was inspired by earlier usage in derived deformation theory, such as \cite{Man2}, and is a natural generalisation of Schlessinger's conditions for set-valued deformation functors from \cite[Theorem 2.11]{Sch}. It differs from the notion of infinitesimal cohesion in \cite{lurie} in that we only require one of the morphisms to be nilpotent, so nowadays homogeneity is frequently referred to as ``infinitesimal cohesion on one factor''.

\begin{definition}\label{Tdef}
Given a homotopy-preserving homogeneous  functor $F:dg_+\Alg(R) \to s\Set$, an object $A \in dg_+\Alg(R)$ with $A_{\#}$ flat over $R_{\#}$  and a point $x \in F(A)$, define the tangent functor
$$
T_x(F/R): dg_+\Mod_{A\ten_RA^{\op}} \to s\Set
$$
by
$$
T_x(F/R)(M):= F(A\oplus M\eps)\by^h_{F(A)}\{x\},
$$
where $\eps$ is central and square-zero, so $A \oplus M\eps$ is given the multiplication $(a_1+m_1\eps)(a_2+m_2\eps):=a_1a_2+a_1m_2\eps +m_1a_2\eps$.
\end{definition}
The reason for the hypothesis that  $A_{\#}$ be flat over $R_{\#}$ is to ensure that $A\ten_RA^{\op}$ is a model for the DGAA $A^{\oL,e}$ of Definition \ref{Aedef}.

\begin{definition}\label{Fcotdef}
 In the setting of Definition \ref{Tdef}, we say that $F$ has a cotangent complex $\bL_{F/R,x}$ at $x$ is there is an $A$-bimodule $ \bL_{F/R,x}$ in chain complexes (possibly incorporating negative degrees) representing $T_x(F/R)$ homotopically  in the sense that the simplicial mapping space 
 \[
  \map_{dg_+\Mod_{A\ten_RA^{\op}}}(\bL_{F/R,x},-) 
 \]
is weakly equivalent to $T_x(F/R)$ when restricted to $dg_+\Mod_{A\ten_RA^{\op}}$. 
\end{definition}
In particular, this means that 
\[
\pi_i T_x(F/R)(M)\cong \EExt^{-i}_{A^{\oL,e}}(\bL_{F/R,x},M)
\]
for all $M \in dg_+\Mod_{A\ten_RA^{\op}}$.
 
\begin{remark}
 Note that the hypotheses ensure that $T_x(F/R)$ takes values in simplicial abelian groups and that  it preserves finite homotopy limits, and in particular loop spaces. In Artinian settings, that suffices to ensure a form of cotangent complex exists, but we have to impose its existence  as an additional condition. 

In the commutative setting, there are finiteness conditions ensuring that cotangent complexes exist (see \cite[Proposition 1.33]{drep}, following  \cite[Theorem 7.4.1]{lurie}.  Adapting this to the non-commutative setting would not be straightforward, but in examples if interest it tends to be relatively easy to construct cotangent complexes by hand. %%Noetherianity
\end{remark}

\begin{lemma}\label{cotgood}
 If $F$ is an  $\infty$-geometric derived  NC Artin prestack, then it has cotangent complexes at all points.%, combining to give a quasi-coherent complex of $\sO_F$-bimodules. These are moreover bounded below in the chain direction.
\end{lemma}
\begin{proof}
We begin with the case where $F$ is sqc $N$-geometric, so arises  as the homotopy colimit of a derived homotopy NC Artin $n$-hypergroupoid $X_{\bt}$. Writing $X_n = \oR \Spec^{nc} O(X)^n$,  \cite[Lemma \ref{stacks2-Lcart} and Corollary \ref{stacks2-cotgood}]{stacks2} adapt to give an expression for the cotangent complex as a homotopy-Cartesian  $O(X)$-bimodule in cosimplicial $R$-modules, which by Yoneda extension determines it as an $\sO_F^{\oL,e}$-module, for the DGAA-valued functor $\sO_F$ on the slice category over $F$ given by $\sO_F(\oR\Spec B \to F):= B$. Explicitly, given a point $x \in X_0(A)$, by adapting \cite[Definition \ref{stacks2-cotdef} and Corollary \ref{stacks2-loopcot}]{stacks2}, we have
\[
 \bL_{F/R,x} \simeq \Tot N_c^{\le N} (\Omega^1_{X/S}\ten_{O(X)^{\oL,e}}^{\oL}A^{\oL,e}),
\]
where $N_c^{\le N}$ denotes $N$-truncated cosimplicial normalisation.

For the general case, we write $F=\LLim_i F_i$ as a filtered colimit of open morphisms of sqc $N$-geometric derived NC prestacks, and then observe that the openness conditions ensure that for all $x \in F_i(A)$, the natural map $   \bL_{F/R,x} \to  \bL_{F_i/R,x}$ is an equivalence, so existence follows from the sqc case.
\end{proof}

\begin{lemma}\label{obs}
Let $F:dg_+\Alg(R) \to s\Set$ be homotopy-preserving and homogeneous, and take  
  a surjection  $f \co A \onto B$ in $ dg_+\Alg_R$ with kernel $I$   and a quotient $C$ of $B$ such that $I \cdot \ker(A \to C)=0$ (so in particular $f$ is a square-zero extension). Then there is an associated obstruction map
\[
o_f \co F(B) \to F(C \oplus I_{[-1]}\eps)
\]
in the homotopy category of simplicial sets, such that 
\[
 F(A) \simeq F(B) \by^h_{o_f,F(C \oplus I_{[-1]}\eps),0} F(C).
\]
\end{lemma}
\begin{proof}
 This is a standard argument (see for instance \cite[Proposition 1.17]{drep}, which takes $B=C$). The idea is that there is an  obvious DGAA structure $\tilde{B}$ on $\cone(I \to A)$,  with $I$ square-zero, so that $\tilde{B}\to B$ is a quasi-isomorphism and hence $F(\tilde{B}) \simeq F(B)$, since $F$ is homotopy-preserving. Then $A= \tilde{B}\by_{C \oplus I_{[-1]}}C$, and the rest follows from homogeneity. 
\end{proof}

\subsubsection{Derived representability}\label{repsn}

We now establish that our derived $n$-geometric  NC prestacks satisfy a weak representability theorem analogous to that of \cite[Appendix C]{hag2}. 
\begin{proposition}\label{derivedatlasprop}
 Take  a derived  NC prestack $F$ over $R$ and a submersive morphism $u \co \Spec A  \to \pi^0F$ from an affine NC $\H_0R$-scheme to the NC prestack $\pi^0F:= F|_{\Alg(\H_0R)}$. Assume  that $F$ is homotopy-preserving (so $F$ fibrant suffices), homogeneous,  and has a cotangent complex at $u \in F(A)$. Assume moreover that $F$ is nilcomplete in the sense that for all $B \in dg_+\Alg(R)$,  applying $F$ to the Postnikov tower of $B$ gives an equivalence
\[
 F(B) \to \ho\Lim_n F(B/\tau_{>n}B).
\]
 %% 
% \begin{enumerate}
%  \item the functor $F$ is homogeneous in the sense that for all nilpotent surjections $B' \onto B$ in $dg_+\Alg(R)$, the natural map
% \[
%  F(B'\by_BC) \to F(B')\by_{F(B)}^hF(C)
% \]
% is a weak equivalence, 
% 
% \item $F$ has a cotangent complex at $u$, in the sense that there exists a bimodule $u^*\oL\Omega^1_{F} \in dg\Mod(A^{\oL,e})$, bounded below in the chain direction, % sure we don't need this?? pro-rep issues if not.
% such that the functor  $\map_{A^{\oL,e}}(x^*\oL\Omega^1_{F},-)$, restricted to objects in  $dg_+\Mod(A^{\oL,e})$, represents the functor
% \[
%   F(A \oplus -)\by^h_{F(A)}\{u\},
% \]
% \item $F$ is nilcomplete in the sense that for all $B \in dg_+\Alg(R)$,  applying $F$ to the Postnikov tower of $B$ gives an equivalence
% \[
%  F(B) \to \ho\Lim_n F(B/\tau_{>n}B).
% \]
% \end{enumerate}
Then there exists a submersive morphism $v \co \oR\Spec \tilde{A} \to F$ with $\pi^0v \cong u$, for some DGAA
$\tilde{A} \in  dg_+\Alg(R)$ with $\H_0\tilde{A}\cong A$. %a derived affine NC $R$-scheme $V$ with $\pi^0V \cong U$,
 
\end{proposition}
\begin{proof}
This proof proceeds exactly as in the commutative case sketched in \cite[Theorem C0.9]{hag2}. We construct $\tilde{A}$ as the homotopy limit of its Postnikov tower, invoking obstruction theory at each stage. In more detail, we inductively construct  an inverse system $\{\tilde{A}(n)\}_n$ in $dg_+\Alg(R)$ with compatible elements $u_n \in F(\tilde{A}(n))$ such that
\begin{enumerate}
 \item $\tilde{A}(0)\simeq A$ with $u_0\simeq u$,
\item the maps $\tilde{A}(n+1) \to \tilde{A}(n)$ induce quasi-isomorphisms $\tilde{A}(n+1)/\tau_{>n}\tilde{A}(n+1) \to \tilde{A}(n)$ (in particular, $\H_i\tilde{A}(n)=0$ for all $i>n$), and
 \item $F$ has a cotangent complex $u_n^*\oL\Omega^1_{F}$ at $u_n$, with the $\tilde{A}(n)^{\oL,e}$-module $\oL\Omega^1_{\tilde{A}(n)/F}:= \cone(\oL\Omega^1_{\tilde{A}(n)} \to  u_n^*\oL\Omega^1_{F})$ being projective to order $n+1$ in the sense that 
\[
 \Ext^i_{\tilde{A}(n)^{\oL,e}}(\oL\Omega^1_{\tilde{A}(n)/F},M)=0
\]
 for all $1\le i \le n+1$ and all $A^{\oL,e}$-modules $M$ concentrated in degree $0$.
\end{enumerate}
Note that via a spectral sequence argument, the third condition is equivalent  to saying that $\Ext^0_{\tilde{A}(n)^{\oL,e}}(\oL\Omega^1_{\tilde{A}(n)/F},N)=0$ for all $\tilde{A}(n)^{\oL,e}$-modules $N$ concentrated in chain degrees $[1,n+1]$. 

The construction and arguments fleshed out in \cite[proof of Theorem 7.1]{PortaYuRep} now adapt to this setting, although (like \cite{hag2}) we have used projectivity rather than flatness, thus avoiding  descent arguments which are not available in our setting. We thus merely outline the main steps.
 
For $n=0$, the conditions are satisfied by setting $\tilde{A}(0):=A$ and $u_0=u$. Given $(\tilde{A}(n),u_n)$, the truncation $\tau_{\le n+1}\oL\Omega^1_{\tilde{A}(n)/F}$ is necessarily projective to order $n+1$, from which it follows that the fibration sequence
\[
(\H_{n+2}\oL\Omega^1_{\tilde{A}(n)/F})_{[-n-2]} \tau_{\le n+2}\oL\Omega^1_{\tilde{A}(n)/F} \to \tau_{\le n+1}\oL\Omega^1_{\tilde{A}(n)/F}
\]
splits. The resulting map $\tau_{\le n+2}\oL\Omega^1_{\tilde{A}(n)/F} \to (\H_{n+2}\oL\Omega^1_{\tilde{A}(n)/F})_{[-n-2]}=:M_{[-n-2]}$ in the homotopy category of $\tilde{A}(n)^{\oL,e}$-modules gives rise to a derivation $\phi$ from $\tilde{A}(n)$, and we set $\tilde{A}(n+1)$ to be the homotopy fibre product
\[
 \tilde{A}(n+1) := \tilde{A}(n)\by^h_{\phi, (A \oplus M_{[-n-2]})}A,
\]
which is a homotopy extension of $\tilde{A}(n)$ by $M_{[-n-1]}$. 

Homogeneity of $F$ then gives rise to a lift $u_{n+1} \in F(\tilde{A}(n+1))$ of $u_n$. Relatively straightforward calculations show that 
\[
 \H_i \oL\Omega^1_{\tilde{A}(n)/\tilde{A}(n+1)} \cong \begin{cases} 0 & i< n+2 \\ \H_{n+2}\oL\Omega^1_{\tilde{A}(n)/F} & i=n+2 \\ 0 & i=n+3, \end{cases}
\]
and combining these with the fibration sequence 
\[
\oL\Omega^1_{\tilde{A}(n+1)/F}\ten_{\tilde{A}(n+1)^{\oL,e}}^{\oL}\tilde{A}(n)^{\oL,e} \to  \oL\Omega^1_{\tilde{A}(n)/F} \to \oL\Omega^1_{\tilde{A}(n)/\tilde{A}(n+1)}  
\]
allows us to conclude that $\oL\Omega^1_{\tilde{A}(n+1)/F}$ is indeed projective to order $n+2$. 

This completes the inductive step, and we now set $\tilde{A}:=\ho\Lim_n \tilde{A}(n)$, with the  element $v \in F(\tilde{A})$ given by the elements $u_n$ via the equivalence $F(\tilde{A})\simeq \ho\Lim_n  F(\tilde{A}(n))$ coming from nilcompleteness.
\end{proof}

 \begin{corollary}\label{repcor}
A homotopy-preserving derived  NC prestack $F$ is $n$-geometric (resp. $\infty$-geometric) if and only if $\pi^0F$ is $n$-geometric (resp. $\infty$-geometric) and
$F$ is nilcomplete, homogeneous, and has a cotangent complex globally.
 \end{corollary}
\begin{proof}
Again, this follows exactly as in the commutative case. The ``only if'' direction is fairly straightforward. %%subtlety with filtered colimits of submersive (open) immersions involved in $\infty$ and non-qc stuff, but OK, because the transition maps are \'etale, so  X(A)\by_{X(\H_0A)}X_{\alpha}(\H_0A) \simeq X_{\alpha}(A)$ 
For the ``if'' direction, we can use the characterisation of \cite[Theorem \ref{stacks2-bigthm}]{stacks2} to see that it suffices to construct an $n$-atlas for $F$ analogous to those of \cite{hag2}. This can be done inductively on $n$, with Proposition \ref{derivedatlasprop} producing $n$-atlases for $F$ from those for $\pi^0F$. 
\end{proof}

\subsubsection{Derived moduli of projective modules}\label{dmodprojmodsn}

\begin{definition}
 Given an associative dg algebra $A$, say that a right $A$-module $M$ in chain complexes is projective if the derived tensor product  $M\ten^{\oL}_A\H_0A$ is quasi-isomorphic to a projective $\H_0A$-module concentrated in degree $0$. 
\end{definition}

\begin{proposition}\label{dprojgeomprop}
The simplicial functor $\bar{W}\cP$ on $dg_+\Alg(R)$ sending an associative dg algebra $A$ to the nerve of the $\infty$-groupoid $\cP(A)$ of finite projective right $A$-modules is a submersive $1$-geometric derived Artin NC prestack.
\end{proposition}
\begin{proof}
 By Proposition \ref{projgeomprop}, we know that $\pi^0F$ is a  submersive $1$-geometric Artin NC prestack. We may then invoke Corollary \ref{repcor}, with reasoning as in \cite[\S 2]{dmsch} showing that $\bar{W}\cP$ is nilcomplete and  homogeneous. That reasoning also  shows that the tangent functor at a point $[P] \in \bar{W}\cP(A)$ corresponding to a projective module $P$ is given by truncating the functor $M \mapsto \oR\HHom_{A^{\op}}(P, P\ten_AM_{[-1]})$. Now
 \[
 \oR\HHom_{A^{\op}}(P, P\ten_AM_{[-1]}) \cong \oR\HHom_{A^{\oL,e}}(P^*\ten^{\oL}_RP,M_{[-1]}),
 \]
 where $P^*:=\oR\HHom_{A^{\op}}(P,A)$, regarded as a left $A$-module, so $\bar{W}\cP$ has a  cotangent complex  $P^*\ten^{\oL}_RP_{[1]}$ at every point $[P]$, and all the conditions of Corollary \ref{repcor} are satisfied. 
\end{proof}

\subsubsection{Derived moduli of perfect complexes}\label{dmodperfsn}

Since the dg category $\per_{dg}(B)$ of perfect complexes is defined for  DGAAs $B$, Definition \ref{Perfdef} extends to give us a derived NC prestack $\Perf\co dg_+\Alg(R) \to s\Set$ of perfect complexes, and we then have the following: 

\begin{proposition}\label{dPerfgeomcor}
The derived NC prestack $\Perf$ is  Artin $\infty$-geometric. 
\end{proposition}
\begin{proof}
 This follows from Corollaries \ref{repcor} and \ref{Perfgeomcor} with identical reasoning to the proof of Proposition \ref{dprojgeomprop}. The cotangent complex at  $[P] \in \Perf(B)$  is again given by the $B^{\oL,e}$-module $\oR\HHom_{B^{\op}}(P,B)\ten^{\oL}_RP_{[1]}$.
\end{proof}

\begin{corollary}\label{dPerfgeomcor2}
 If $S \in \Alg(R)$ is finite and flat as an $R$-module, then the  simplicial functor on $dg_+\Alg(R)$ given by $\Perf(-\ten_RS)$ is Artin $\infty$-geometric.
\end{corollary}
\begin{proof}
This follows by combining Proposition \ref{dPerfgeomcor} with Lemma \ref{stweillemma2}. 
%combined with the observation that finiteness and flatness of $S$ over $R$ imply that  for an NC affine scheme $X$, the functor $A \mapsto X(S\ten_RA)$ is another NC affine scheme (we can think of this as a non-commutative analogue of Weil restriction of scalars), and that this functor respects submersive morphisms, epimorphisms and open morphisms. 
\end{proof}

\begin{remark} \label{Perfnicermk}
 It is natural to ask whether we can extend Corollary \ref{dPerfgeomcor2} by taking more general DGAAs or even dg categories  in place of $S$. Constructing atlases is very hard work, but  the reasoning of \cite[\S 2]{dmsch} still ensures that for any dg category $\cA$ over $R$, the derived NC prestack $\Perf_{\cA}:= \Perf(\cA\ten_R^{\oL}-)$  satisfies most of the conditions of  Proposition \ref{derivedatlasprop}, being homogeneous and nilcomplete. 
 
 The tangent functor of $\Perf_{\cA}$ at the point $[P]\in \Perf(\cA\ten_R^{\oL}B)$  corresponding to a perfect $\cA\ten_R^{\oL}B$-module $P$ is given on $B^{\oL,e}$-modules by 
 \[
 M \mapsto \oR\HHom_{\cA \ten_R^{\oL} B}(P, P\ten_B^{\oL}M)_{[-1]} \simeq \oR\HHom_{\cA \ten_R^{\oL} B}(P, P \ten_R^{\oL}B)_{[-1]}\ten_{B^{\oL,e}}^{\oL}M.
 \]
 Thus the cotangent complex at $[P]$ is given by the dual of the $B^{\oL,e}$-module $\oR\HHom_{\cA \ten_R^{\oL} B}(P, P \ten_R^{\oL}B)_{[-1]}$ whenever the latter is perfect, where we regard $P \ten_R^{\oL}B$ as a $B$-module with respect to the right action of $B$ on itself, and as a $B^{\oL,e}$-module by combining the left action on $B$ with the right action on $P$. 
 
 We can then deduce that for perfect cotangent complexes to exist at all points $[P]$ of $\Perf_{\cA}$, it suffices for $\cA$ to be a locally proper dg category over $R$, meaning that the morphism complexes $\cA(X,Y)$ should all be perfect complexes of $R$-modules. This follows because 
 for any objects $X,Y \in \cA$, we know that 
 \begin{align*}
  \oR\HHom_{\cA \ten_R^{\oL} B}((X\ten^{\oL}_RB), (Y\ten^{\oL}_RB) \ten_R^{\oL}B)&\simeq \oR\HHom_{\cA}(X,Y)\ten^{\oL}_RB^{\oL,e},
 \end{align*}
which is perfect over $B^{\oL,e}$, and that $P$ is a retract of a finite extension of modules of the form $X\ten^{\oL}_RB$.

The absence of descent in this NC setting means it is not clear whether hypergroupoid presentations do in fact exist for $\Perf_{\cA}$. However, Lurie's representability theorem in its simplified form as \cite[Theorem 2.17%\ref{drep-lurierep3}
]{drep} (also see \cite[Theorem 4.12%\ref{dmsch-representdmod}
]{dmsch}) can easily be applied to give presentations for the commutative restriction $\Perf_{\cA}^{\comm}$ once one permits \'etale descent. Explicit hypergroupoid presentations for closely related moduli problems appear in \cite{benzeghliGeometricityMC,benzeghliGP}. 
 \end{remark}

\section{Stacky thickenings of non-commutative derived affines and formal stacks}\label{stackysn}

The shifted double Poisson structures in \cite{NCpoisson} will only satisfy \'etale functoriality, meaning that of the objects we have encountered so far, they cannot be defined on anything more general than derived NC Deligne--Mumford prestacks. Instead of using DGAAs as building blocks, the solution is to introduce a non-commutative analogue of the stacky CDGAs of \cite[\S \ref{poisson-stackyCDGAsn}]{poisson}, which we call stacky DGAAs. These will behave like formal completions of Artin prestacks, and should be thought of as non-commutative analogues of Lie algebroids (unfortunately, associative algebroid already has quite a different meaning). We will then see that derived  NC  Artin prestacks admit \'etale hypercovers by such stacky DGAAs.

The lack of descent in our non-commutative setting also  means that the non-commutative $n$-geometric prestacks of the previous section might not be  as prevalent as we might wish. We would for instance like to study moduli of perfect complexes over proper dg categories as in Remark \ref{Perfnicermk}. However, we will see that for such derived NC prestacks admitting cotangent complexes everywhere, there do exist resolutions by  \'etale stacky DGAAs, which will then permit  well-behaved Poisson structures to be defined and constructed on very general NC moduli functors  in \cite{NCpoisson}.

\subsection{Stacky thickenings of derived affines}

We now adapt some definitions and lemmas from \cite[\S \ref{poisson-Artinsn}]{poisson}. %, as summarised in \cite[\S \ref{DQvanish-bicdgasn}]{DQvanish}. 
Recall that
we   regard the  DGAAs of \S \ref{NCmodulisn}   as chain complexes $\ldots \xra{\delta} A_1 \xra{\delta} A_0 \xra{\delta} \ldots$ rather than cochain complexes --- this will enable us to distinguish easily between derived (chain) and stacky (cochain) structures.  

\begin{definition}
A stacky DGAA over a CDGA $R_{\bt}$ is  an associative  $R$-algebra $A^{\bt}_{\bt}$ in  cochain chain complexes. We write $DGdg\Alg(R)$ for the category of  stacky DGAAs over $R$, and $DG^+dg\Alg(R)$ (resp. $DG^+dg_+\Alg(R)$) for the full subcategory consisting of objects $A$ concentrated in non-negative cochain degrees (resp. non-negative bidegrees).
\end{definition}

When working with chain cochain complexes $V^{\bt}_{\bt}$, we will usually denote the chain differential by $\delta \co V^i_j \to V^i_{j-1}$, and the cochain differential by $\pd \co V^i_j \to V^{i+1}_j$.

\begin{definition}
 Say that a morphism $U \to V$ of chain cochain complexes is a levelwise quasi-isomorphism if $U^i \to V^i$ is a quasi-isomorphism of chain complexes for all $i \in \Z$. Say that a morphism of stacky DGAAs is a levelwise quasi-isomorphism if the underlying morphism of chain cochain complexes is so.
\end{definition}

The following is a consequence of \cite[Theorem 11.3.2]{Hirschhorn}, with essentially the same proof as \cite[Lemma \ref{poisson-bicdgamodel}]{poisson}:
\begin{lemma}\label{bidgaamodel}
There is a cofibrantly generated model structure on stacky DGAAs over $R$ in which fibrations are surjections and weak equivalences are levelwise quasi-isomorphisms. 
\end{lemma}

There is a Dold--Kan denormalisation functor $D$ from non-negatively graded DGAAs to cosimplicial associative algebras; the explicit formulae of  \cite[Definition \ref{ddt1-nabla}]{ddt1} are still valid in the non-commutative setting. This necessarily has a left adjoint, which we denote by $D^*$; for an explicit description, note that the formula of \cite[Definition \ref{DQDG-Dstardef}]{DQDG} is still valid in this more general case. For most practical purposes, the functor $D^*$ can be understood by remembering that it sends the tensor algebra (tensor products taken levelwise) on a cosimplicial space $V$  to the tensor algebra (with graded tensor products) on the cosimplicial normalisation $N_cV$ (given by $(N_cV)^n:= \{v \in V^n ~:~ \sigma^iv=0 ~\forall i\}$ with differential $\pd= \sum_i (-1)^i \pd^i$).

%now describe explicitly, adapting \cite[Definition \ref{DQDG-Dstardef}]{DQDG}. 

% \begin{definition}\label{Dstardef}
% Given a cosimplicial associative algebra $A$, we define the associative cochain algebra $D^*A$ as follows. We first consider the normalised cochain complex $NA$ given by
% \[
%  N^mA:= \{a \in A^m ~:~ \sigma^ja = 0 \in A^{m-1}, ~\forall~ 0 \le i <m\},
% \]
% with differential $ \pd a:= \sum(i-1)^i \pd^ia$. We then define an associative  product $\smile$ (a variant of the Alexander--Whitney cup product) on $NA$ by 
% \[
% a \smile b := (\pd^{[m+1,m+n]}a)\cdot (\pd^{[1,m]}b)  
% \]
% for  $a \in N^mA$, $b \in N^nA$. 
% 
% The associative  cochain algebra  $D^*A$ is then  the quotient of $NA$ by the relations
% \[
% (\pd^Ia)\cdot (\pd^J b) \sim \left\{ \begin{matrix} (-1)^{(J, I)} (a\smile b) & a \in A^{|J|},~ b\in A^{ |I|},\\ 0 & \text{ otherwise},\end{matrix} \right.
% \]
% for (possibly empty) sets $I,J$ with $I \cap J= \emptyset$,
% where for disjoint sets $S,T$ of integers, $(-1)^{(S,T)}$ is the sign of the shuffle permutation of $S \sqcup T $ which sends the first $|S|$ elements to $S$ (in order), and the remaining $|T|$ elements to $T$ (in order). 
% \end{definition}

For any cosimplicial chain DGAA $A$, we then have a stacky DGAA $D^*A$ concentrated in non-negative cochain degrees. The proof of 
 \cite[Lemma \ref{poisson-Dstarlemma}]{poisson} adapts to show that  
$D^*$ is a left Quillen functor from the Reedy model structure on cosimplicial chain DGAAs to the model structure of Lemma \ref{bidgaamodel}.

Since   $DA$ is a pro-nilpotent extension of $A^0$, when $A \in DG^+dg_+\Alg(R)$ we think of the simplicial presheaf  $\oR \Spec^{nc} DA$ as a stacky derived thickening of the non-commutative derived affine scheme $\oR \Spec^{nc} A^0$. 

\begin{example}\label{DstarBGex}
Given $A \in dg_+\Alg(R)$, we can consider the derived version $[\oR \Spec^{nc} A/\oR\bG_m] $ of the  NC prestack  $[\Spec^{nc} A/\bG_m] $ of Example \ref{repnex}, where $\bG_m$ acts by conjugation. In the model category of derived NC prestacks, this is the homotopy colimit of a simplicial derived NC affine stack given in simplicial degree $n$ by  $\oR \Spec^{nc} A\by \oR\bG_m^n$, so is associated to a cosimplicial chain DGAA $B$ given in cosimplicial degree $n$ by $A\<t_1^{\pm}, \ldots, t_n^{\pm}\>$. %, with cosimplicial operations fixing $A$ except for $\pd^0(a)= t_1^{-1}at_1$, while $\pd^0(t_1)= t_2$, $\pd^1(t_1)= t_1t_2$,   $\pd^2(t_1)= t_2$,  and $\sigma^0(t_1)= 1$.

Then $D^*B$ is isomorphic to the stacky DGAA $A\<s\>$ for $s \in (D^*B)^1_0$, with $\pd s = s^2$ and $\pd a = sa-as$, for $a \in A$. We can loosely think of this as representing $[\Spec^{nc} A/\g_m]$, for $\g_m$  an infinitesimal neighbourhood of $1 \in \bG_m$. In this non-commutative context, such formal group schemes $\g$ correspond to non-unital associative algebras rather than Lie algebras, and $\g_m =R$ with its usual multiplication (whereas the formal neighbourhood $\g_a$ of the additive group $\bG_a$ is $R$ with zero multiplication). 

In general, there is a similar bar construction $[\Spec^{nc} A/\g]$  whenever a non-unital associative algebra $\g$ acts on $A$ in the form of a morphism $\alpha \co \g \to \Der(A, A\ten A)$ satisfying $\alpha(uv)(a)'\ten 1 \ten \alpha(uv)(a)''= \alpha(u)(\alpha(v)(a)')\ten \alpha(v)(a)'' + \alpha(u)' \ten \alpha(v)(\alpha(u)(a)'')$ in sumless Sweedler notation.
\end{example}

\begin{definition}
 Given a chain cochain complex $V$, define the cochain complex $\hat{\Tot} V \subset \Tot^{\Pi}V$ by
\[
 %(\hat{\Tot} V)^m := (\prod_{i \le 0} V^i_{i-m}) \oplus (\bigoplus_{i>0}   V^i_{i-m}),
(\hat{\Tot} V)^m := (\bigoplus_{i < 0} V^i_{i-m}) \oplus (\prod_{i\ge 0}   V^i_{i-m})
\]
with differential $\pd \pm \delta$. 
\end{definition}
The key property of the semi-infinite total complex $\hat{\Tot}$ is that it sends levelwise quasi-isomorphisms in the chain direction to quasi-isomorphisms; the same is not true in general of the sum and product total complexes $\Tot, \Tot^{\Pi}$, cf. \cite[\S 5.6]{W}.
% The functor $\hat{\Tot}$ is referred to as Tate realisation in \cite{CPTVV}.

\begin{definition}
 Given a stacky DGAA $A$ and $A$-modules $M,N$ in chain cochain complexes, we define the chain cochain complex 
$\cHom_A(M,N)$  by 
\[
 \cHom_A(M,N)^i_j=  \Hom_{A^{\#}_{\#}}(M^{\#}_{\#},N^{\#[i]}_{\#[j]}),
\]
with differentials  $\pd f:= \pd_N \circ f \pm f \circ \pd_M$ and  $\delta f:= \delta_N \circ f \pm f \circ \delta_M$,
where $V^{\#}_{\#}$ denotes the bigraded vector space underlying a chain cochain complex $V$. 

We then define the  $\Hom$ complex $\hat{\HHom}_A(M,N)$ by
\[
 \hat{\HHom}_A(M,N):= \hat{\Tot} \cHom_A(M,N).
%\HHom_A(M,N):= \Tot^{\Pi} \tau^{\le 0}\cHom_A(M,N)
\]
\end{definition}
Note that $\hat{\Tot}$ is lax monoidal with respect to tensor products, which means in particular that   there is a multiplication $\hat{\HHom}_A(M,N)\ten_R \hat{\HHom}_A(N,P)\to \hat{\HHom}_A(M,P)$  (the same is not true for $\Tot^{\Pi} \cHom_A(M,N)$ in general).

\begin{definition}\label{Omegastackydef}
 Given  a morphism $C \to A$ in $DGdg\Alg(R)$, we define the $A^{e}$-module $\Omega^1_{A/C}$ in double complexes to be the kernel of the multiplication map $A\ten_C A \to A$, and we  denote its differential (inherited from $A$) by $\delta$. 

We denote by $\oL\Omega^1_{A/C}$ the cotangent complex, given by 
\[
 \cocone(A\ten^{\oL}_CA \to A),
\]
regarded as an $A^{\oL,e}$-module in double complexes, where the derived tensor product is taken with respect to levelwise quasi-isomorphisms, and cocone is taken in the chain direction.
\end{definition}

\begin{remarks}\label{cotstackyremarks}
Observe that for an $A$-bimodule $M$, an $A$-bilinear map $\Omega^1_{A/C} \to M$ is essentially the same thing as an $C$-bilinear derivation $A\to M$, via the universal derivation $d \co A \to \Omega^1_{A/C}$ given by $da=a\ten 1 - 1\ten a$. %%inverse $a\ten b \mapsto adb$.

Note that if $C \to A$ is a morphism in $DGdg_+\Alg(R)$, then  $\oL\Omega^1_{A/C}$   is quasi-isomorphic to 
 $A\ten_{\tilde{A}}\Omega^1_{\tilde{A}/C}\ten_{\tilde{A}}A$ for any factorisation $C \to \tilde{A} \to A$ with $\tilde{A} \to A$ a quasi-isomorphism and $\tilde{A}_{\#}^{\#}$ flat as a left or right $C_{\#}^{\#}$-module; in particular this applies if $\tilde{A}$ is a cofibrant replacement of $A$ over $C$. %%dg_+ needed for Spaltenstein purposes.
\end{remarks}

Writing $\Omega^1_A:= \Omega^1_{A/R}$, we have:
\begin{definition}\label{hetdef}
 A morphism  $A \to B$ in $DG^+dg\Alg(R)$ is said to be  homotopy \'etale when the maps 
\[
  (\oL\Omega_{A}^1\ten_{A^{\oL,e}}^{\oL}(B^{\oL,e})^0)^i \to (\oL\Omega_{B}^1\ten_{B^{\oL,e}}^{\oL}{B^{\oL,e}}^0)^i
\]
are quasi-isomorphisms for all $i \gg 0$, and 
\[
\Tot \sigma^{\le q} (\oL\Omega_{A}^1\ten_{A^{\oL,e}}^{\oL}(B^{\oL,e})^0) \to \{\Tot \sigma^{\le q}(\oL\Omega_{B}^1\ten_{B^{\oL,e}}^{\oL}{B^{\oL,e}}^0)
\]
is a quasi-isomorphism for all $q \gg 0$, where $\sigma^{\le q}$ denotes the brutal cotruncation
\[
 (\sigma^{\le q}M)^i := \begin{cases} 
                         M^i & i \le q, \\ 0 & i>q.
                        \end{cases}
\]
\end{definition}

% \begin{definition}\label{hfetdef}
%  A morphism  $A \to B$ in $DG^+dg\Alg(R)$ is said to be  homotopy formally \'etale when the map
% \[
%  \{\Tot \sigma^{\le q} (\oL\Omega_{A}^1\ten_{A^{\oL,e}}^{\oL}(B^{\oL,e})^0)\}_q \to \{\Tot \sigma^{\le q}(\oL\Omega_{B}^1\ten_{B^{\oL,e}}^{\oL}{B^{\oL,e}}^0)\}_q
% \]
% is a pro-quasi-isomorphism (i.e. an essentially levelwise quasi-isomorphism in the sense of \cite[\S 2.1]{isaksenStrict}), where $\sigma^{\le q}$ denotes the brutal cotruncation
% \[
%  (\sigma^{\le q}M)^i := \begin{cases} 
%                          M^i & i \ge q, \\ 0 & i<q.
%                         \end{cases}
% \]
% \end{definition}
% This is equivalent to saying that the system $\{\Tot \sigma^{\le q} (\oL\Omega_{B/A}^1\ten_{B^{\oL,e}}^{\oL}(B^{\oL,e})^0)\}_q$ is  essentially levelwise quasi-isomorphic to $0$ in the sense of \cite[\S 2.1]{isaksenStrict}.

\subsubsection{Modules over stacky DGAAs}\label{modDGAAsn}

The following is adapted from the corresponding results for stacky CDGAs in \cite[\S 3.1.1]{poisson}.

\begin{definition}\label{almostdef}
 As for instance in \cite[Definition \ref{stacks2-delta*}]{stacks2}, define almost cosimplicial diagrams to be functors on the subcategory  $\Delta_*$  of the ordinal number category $\Delta$ containing only  those morphisms $f$ with $f(0)=0$; define almost simplicial diagrams dually. Thus an almost simplicial diagram $X_*$ in $\C$ consists of objects $X_n \in \C$, with all of the operations $\pd_i, \sigma_i$ of a simplicial diagram except $\pd_0$,  satisfying the usual relations. 
\end{definition}

Given a simplicial (resp. cosimplicial) diagram $X$, we write $X_{\#}$ (resp. $X^{\#}$) for the underlying almost simplicial  (resp. almost cosimplicial) diagram.

The denormalisation functor $D$ descends to a functor from graded associative algebras to almost cosimplicial algebras, with $D^*$ thus descending to a functor in the opposite direction. In other words, $(D^*B)^{\#}$ does not depend on $\pd^0_B$, and  $\pd^0_{DA}$ is the only part of the structure on  $DA$ to depend on $\pd_A$. The same is true for the left adjoint $D^*_{\Mod}$ of the denormalisation functor $D$ from cochain $A$-modules to  cosimplicial $DA$-modules. 

The following is now an immediate consequence of the  Dold--Kan correspondence for almost cosimplicial $A^0$-modules.  
\begin{lemma}\label{denormmod}
Take $A \in DG^+dg\Alg(\Q)$ and a right $DA$-module $M$ in cosimplicial chain complexes, such that the underlying   almost cosimplicial (see Definition \ref{almostdef}) graded module $M^{\#}_{\#}$ is isomorphic to $D^{\#}A_{\#}\ten_{A^0_{\#}}L_{\#}$ for an almost cosimplicial graded $A^0_{\#}$-module $L$. Then 
\[
 D^*_{\Mod}(M)^{\#}_{\#} \cong A^{\#}_{\#}\ten_{A^0_{\#}}N_cL_{\#}
\]
as graded $A^{\#}_{\#}$-modules, where $N_c$ denotes Dold--Kan conormalisation. 
\end{lemma}

The following is \cite[Lemma \ref{poisson-Homrepmod}]{poisson}, which adapts to the non-commutative setting with exactly the same proof:
\begin{lemma}\label{Homrepmod}
 For $A \in DG^+dg\Alg(\Q)$,  a  levelwise cofibrant $DA$-module $M $ in $cdg\Mod(DA)$, and  $P \in DGdg\Mod( A)$, there is a canonical quasi-isomorphism
\[
 \oR \HHom_{DA}(M, DP) \simeq \Tot^{\Pi}\sigma^{\ge 0}\cHom_{A}(D_{\Mod}^*M, P).
\]
\end{lemma}

\begin{definition} \label{Cartesiandef}
 Given a stacky DGAA $A\in DG^+dg\Alg(\Q)$, say that a right $A$-module $M$ in cochain chain complexes is homotopy-Cartesian if the natural maps
\[
 A^i\ten^{\oL}_{A^0}M^0 \to M^i
\]
 are quasi-isomorphisms of chain complexes for all $i$. Similarly, we say that an $A$-bimodule is homotopy-Cartesian if it is homotopy-Cartesian as a module over the stacky DGAA $A^{\oL,e}= A\ten^{\oL}_R A^{\op}$.
\end{definition}

If  $C$ is a cosimplicial diagram of DGAAs, and  $M$ a right $C$-module which is homotopy-Cartesian in the sense that   $\pd^i_M \co M^n\ten^{\oL}_{C^n,\pd^i_C}C^{n+1} \to M^{n+1}$ is a  quasi-isomorphism for all $\pd^i$, then 
the map 
\[
 (\oL\eta^*M)^0\ten_{D^*C^0}^{\oL}D^{\#}D^*C \to (\oL\eta^*M)^{\#} 
\]
 of almost cosimplicial chain complexes  induced by  the unit $\eta \co C \to DD^*C$ of the adjunction   is a levelwise quasi-isomorphism. 
% Observe that the Dold--Kan correspondence for almost cosimplicial $D^*O(X)^0$-modules then gives 
% \[
%  \oL D_{\Mod}^*((\oL\eta^*M)^0\ten_{D^*O(X)^0}^{\oL}D^{\#}D^*O(X)) \simeq (\oL\eta^*M)^0\ten_{D^*O(X)^0}^{\oL}N_cD^{\#}D^*O(X),
% \]
%  
Applying $\oL D^*_{\Mod}$, it then follows from Lemma \ref{denormmod} 
that the map
$(\oL\eta^*M)^0\ten_{D^*C^0}^{\oL} D^*C^{\#} \to \oL D^*_{\Mod}(\oL\eta^*M)^{\#}$ is a levelwise quasi-isomorphism, so $\oL D^*_{\Mod}(\eta^*M)^{\#}$ is also homotopy-Cartesian.

\subsection{Formal representability for derived NC  prestacks}\label{formalrepsn}

\subsubsection{Functors on stacky DGAAs}

\begin{definition}\label{Dlowerdef}
 Given a functor $F:dg_+\Alg(R) \to s\Set$, we define a functor $D_*F$ on  $DG^+dg_+\Alg(R)$ as the homotopy limit
\[
 D_*F(B):= \ho\Lim_{n \in \Delta} F(D^nB),
\]
for the cosimplicial denormalisation functor $D\co DG^+dg_+\Alg(R) \to dg_+\Alg(R)^{\Delta} $  (cf. \cite[Definition \ref{ddt1-nabla}]{ddt1}).
\end{definition}
Thus a model for $D_*F$ is the derived total space
\[
 \oR \Tot F(D^{\bt}B)  =\{ x \in \prod_n \oR F(D^nB)^{\Delta^n}\,:\, \pd^i x_n = \pd_i^{\Delta}x_{n+1},\,\sigma^i x_n = \sigma_i^{\Delta}x_{n-1}\},
\]
of
\cite[\S VIII.1]{sht}, where $\oR F(D^{\bt}B) $ is a  Reedy fibrant replacement of the cosimplicial space $F(D^{\bt}B) $,  the simplicial sets $Y^{\Delta^n}$ are given by $(Y^{\Delta^n})_m:= \Hom_{s\Set}(\Delta^m \by \Delta^n, Y)$,   and $\pd i^{\Delta}, \sigma_i^{\Delta}$ are defined in terms of the face and degeneracy maps between the simplices $\Delta^n$. 

\begin{example}\label{DstarPerf}
We can apply this definition to the derived NC moduli prestacks $\Perf$  of perfect complexes from  Definition \ref{Perfdef} and $\cP$ of finite projective modules from Proposition \ref{dprojgeomprop}.

 Given a cosimplicial DGAA $C \in dg_+\Alg(R)^{\Delta}$,  the space $\ho\Lim_{n \in \Delta}\Perf(C^n)$ is equivalent to the space of homotopy-Cartesian perfect right $C$-modules. It thus follows from \S \ref{modDGAAsn} (also see \cite[\S \ref{smallet-qcohsn}]{smallet}) that $D_*\Perf(B)$ is equivalent to the space of homotopy-Cartesian right $B$-modules $P$ in double complexes  for which $P^0$ is perfect over $B^0$, with equivalences defined levelwise in the chain direction. When $B$ is the stacky DGAA $\Omega^{\bt}_A$ of Definition \ref{Omegabtdef}, these objects are equivalent to perfect $A$-modules with flat non-commutative connections, in a suitably homotopy-coherent sense.

Similarly,  $D_*\cP(B)$ is equivalent to the space of homotopy-Cartesian right $B$-modules $P$ in double complexes  for which the right $B^0$-module $P^0$ is finite and projective. 

This construction thus gives an efficient characterisation of NC derived moduli of real local systems on a manifold $X$, as $B \mapsto D_*\cP(A^{\bt}(X)\ten_{\R}B)$, where $A^{\bt}(X)$ is the de Rham complex of infinitely differentiable forms on $X$, regarded as a stacky CDGA concentrated in chain degree $0$. This works because projective $A^0(X)$-modules correspond to finite rank vector bundles on $X$, and the additional structure gives the data of an  $\infty$-connection, along the lines of \cite{BlockSmith}. Similarly,   $B \mapsto D_*\Perf(A^{\bt}(X)\ten_{\R}B)$ gives derived moduli of complexes of sheaves with finite-dimensional locally constant homology. %%justification that locally we're looking at derived cat of v.s., so complex we construct has loc const homology.
In this case, commutativity of $A^{\bt}(X)$ manifests itself in a symmetric lax monoidal structure on these functors $F$, with $F(B)\by F(C) \to F(B\ten_{\R}C)$ given by $(\sE,\sE')\mapsto \sE\ten_{A^{\bt}(X)}\sE'$. 
\end{example}

\begin{lemma}\label{DFtgtlemma}
Take a stacky DGAA $B \in DG^+dg_+\Alg(R) $, a $B^0$-bimodule $M$ in double complexes, and a homotopy-preserving homogeneous functor $F\co dg_+\Alg(R) \to s\Set$ with a cotangent complex $L$ at $x \in F(B^0)$ in the sense of Definition \ref{Fcotdef}. Then for any $\tilde{x}\in D_*F(B)$ with image $x \in F(B^0)$, we  have a natural equivalence
\[
 D_*F(B\oplus M\eps) \by^h_{D_*F(B)}\{\tilde{x}\} \simeq \map_{dg\Mod(B^0)}(L, \Tot^{\Pi}N_cM),
\]
where $\eps$ is central and square-zero, $\Tot^{\Pi}$ denotes the product total complex, and $N_c$ the cosimplicial normalisation. 
\end{lemma}
\begin{proof}
Since we have $B\oplus M\eps = (B^0\oplus M\eps)\by_{B^0}B$ with $D^iB \to B^0$ a nilpotent (in fact $i$-nilpotent) surjection, we have 
 \begin{align*}
 D_*F(B\oplus M\eps) \by^h_{D_*F(B)}\{\tilde{x}\} &\simeq D_*F(B^0\oplus M\eps) \by^h_{F(B^0)}\{x\}\\
 &\simeq  \ho\Lim_{n \in \Delta} F(B^0\oplus D^nM\eps)\by^h_{F(B^0)}\{x\}\\
 &\simeq  \ho\Lim_{n \in \Delta} \map_{dg\Mod(B^0)}(L, D^nM)\\
 &\simeq   \map_{dg\Mod(B^0)}(L, \ho\Lim_{n \in \Delta}D^nM)\\
& \simeq \map_{dg\Mod(B^0)}(L, \Tot^{\Pi}N_cM). \qedhere
 \end{align*}
\end{proof}

% \begin{lemma}
%  Given a %homotopy-preserving homogeneous  
%  functor $F:dg_+\Alg(R) \to s\Set$ and $B \in DG^+dg_+\Alg(R) $, the natural map 
%  \[
%   D_*F(B) \to \ho\Lim_n D_*F(B^{\le n})
%  \]
% is a weak equivalence.
% \end{lemma}
% \begin{proof}
% Observe that the DGAA $D^iB$ 
% depends only on $B^{\le i)$,  so $F(D^iB)\simeq  \ho \Lim_n F(D^iB^{\le n})$ since the limit stabilises. Thus 
% \[
%  D_*F(B)\simeq  \ho\Lim_{i \in \Delta} \ho \Lim_n F(D^iB^{\le n})\simeq \ho \Lim_n\ho\Lim_{i \in \Delta}  F(D^iB^{\le n}),
% \]
% as required.
% \end{proof}

\begin{definition}\label{Omegabtdef}
 Given $B \in DG^+dg_+\Alg(R)$ (or more generally in ``$G^+dg_+\Alg(R)$'' as the construction is independent of $\pd$), define $\Omega^{\bt}_B \in DG^+dg_+\Alg(R)$ to be the bigraded associative algebra given by the tensor algebra construction
 \[
  (\Omega^{\bt}_B)^{\#}:= \bigoplus_{p \ge 0} (\Omega^p_B)^{[-p]} =  \bigoplus_{p \ge 0} (\underbrace{\Omega^1_B\ten_B\Omega^1_B\ten_B\ldots \ten_b\Omega^1_B} )^{[-p]},
 \]
equipped with the de Rham differential $\pd:=d \co \Omega^p_B \to \Omega^{p+1}_B$ in the cochain direction, and the usual structural differentials $\delta$ in the chain direction.
 \end{definition}
 Note that we could alternatively characterise this construction by the property that it is left adjoint to the functor which forgets the cochain differential $\pd$.

\begin{proposition}\label{cosimplicialresnprop}
 If a homotopy-preserving homogeneous  functor $F:dg_+\Alg(R) \to s\Set$ has a cotangent complex at a point  $x \in F(A)$ which is homologically bounded below, then the functor 
 \[
\hat{F}_x \co   B \mapsto D_*F(B)\by_{F(B^0),x^*}^h\map_{dg_+\Alg(R)}(A,B^0)
 \]
on the category $DG^+dg_+\Alg(R) $ 
is representable by a cosimplicial homotopy  \'etale diagram $C(-) \co \Delta \to DG^+dg_+\Alg(R) $, in the sense that
\[
 \hat{F}_x(B) \simeq \ho\LLim_{j \in \Delta} \map_{DG^+dg_+\Alg(R)}(C(j),B),
\]
naturally in $B$. %%and can replace $\map$ with $\Hom$ --- don't bother?
\end{proposition}
\begin{proof}
  We will construct the representing object $C(-)$ by an inductive process, as a filtered colimit $\LLim_i C^{(i)}(-)$, with $C^{(i)}(-)$ representing the functor $\hat{F}_x^{(i)}\co B \mapsto D_*F(B^{\le i})\by_{F(B^0),x^*}^h\map(A,B^0)$. Since the DGAA $D^iB$ 
 depends only on $B^{\le i}$, we automatically have  
 \begin{align*}
  D_*F(B) &\simeq \ho\Lim_{n \in \Delta} F(D^nB)\\
  &\simeq   \ho\Lim_{n \in \Delta}\ho\Lim_m F(D^n(B^{\le m}))\\
   &\simeq   \ho\Lim_m\ho\Lim_{n \in \Delta} F(D^n(B^{\le m}))\\
   &\simeq \ho\Lim_m  D_*F(B^{\le m}),
 \end{align*}
%$D_*F(B) \simeq \ho\Lim_n D_*F(B^{\le n})$, the inverse system stabilising on each term, 
 and hence $\hat{F}_x(B) \simeq  \ho\Lim_i \hat{F}_x^{(i)}$.
 
Without loss of generality, we may assume that $A$ is cofibrant. Then $\hat{F}_x^{(0)}( B)$ is  represented by the non-commutative de Rham double complex $\Omega^{\bt}_{A/R}$ of Definition \ref{Omegabtdef}.
%, i.e.  the tensor algebra of the bimodule $(\Omega^1_A)^{[-1]}$ (see Definition \ref{Omegadef}) over $A$, equipped with the de Rham differentials $d \co \Omega^p_{A/R} \to \Omega^{p+1}_{A/R}$ in addition to the structural chain differentials $\delta$; the de Rham differential $d$ is determined by the properties that $d\circ d=0$ and $da=a\ten 1 - 1\ten a$ for $a \in A$. 
We thus set $C^{(0)}(j)=\Omega^{\bt}_A$ for all $j \in \Delta$. 

Before proceeding further, observe that since the cotangent complex of $F$ at $x$ exists and is homologically bounded below, we may choose a model $L\in dg\Mod(A)$ for it which is cofibrant and strictly bounded below, so there exists some $M\ge 0$ such that $L_i=0$ for all $i<-M$. 

Now assume that we have constructed $C^{(n-1)}(-)$ with $(C^{(n-1)})^0(-)$ the constant functor $A$. We now wish to represent $\hat{F}_x^{(n)}$, which we can rewrite as $\hat{F}_x^{(n)}(B) =  D_*F(B^{\le n})\by_{F(B^{\le n-1})}^h\hat{F}_x^{(n-1)}(B)$. %\map(C(n-1),B) $.
Applying  Lemma \ref{obs}, we can then write 
\[
D_*F(B^{\le n})\simeq D_*F(B^{\le n-1})\by^h_{o_e, D_*F(B^0 \oplus (B^n)^{[-n]}_{[-1]})}F(B^0)
\]
for some obstruction map $o_e\co D_*F(B^{\le n-1}) \to D_*F(B^0 \oplus (B^n)^{[-n]}_{[-1]})$  associated to the square-zero extensions $D^ie \co D^i(B^n) \to D^i(B^{\le n}) \to D^i(B^{\le n})$. This in turn induces an obstruction map
\[
 \hat{F}_x^{(n-1)}(B)\to D_*F(B^0 \oplus (B^n)^{[-n]}_{[-1]})\by_{F(B^0),x^*}^h\map(A,B^0),
\]
so %letting $L \in dg\Mod(A)$ be a cofibrant model for the cotangnet complex of $F$ at $x$, 
we may apply   Lemma \ref{DFtgtlemma} to rewrite this  as
\[
 \hat{F}_x^{(n-1)}(B)\by^h_{\map_{dg_+\Alg(R)}(A,B^0)}\{f\} \to \map_{dg\Mod(A)}(L, f_*B^n_{[n-1]}),
\]
with homotopy fibre $\hat{F}_x^{(n)}(B) \by^h_{\map_{dg_+\Alg(R)}(A,B^0)}\{f\} $ over $0$.

Since $C^{(n-1)}(-)$ represents $\hat{F}_x^{(n-1)}$, we have
%%Idea is that $\hat{F}_x^{(i)}=\ho\Lim_{j \in \Delta} \oR\Spec (C(i))^j$, so diagram to the limit  gives a corresponding elt of $\ho\Lim_{j \in \Delta}\hat{F}_x^{(n-1)}((C(n-1))^j)$. 
a universal element 
\[
u \in \ho\Lim_{j \in \Delta}\hat{F}_x^{(n-1)}(C^{(n-1)}(j))
\]
mapping to the identity in 
\[
\ho\Lim_{j \in \Delta}\map_{dg_+\Alg(R)}(A,C^{(n-1)}(j)^0)= \map_{dg_+\Alg(R)}(A,A).
\]
The obstruction above therefore gives us an element
\[
 o_e(u) \in \ho\Lim_{j \in \Delta}\map_{dg\Mod(A)}(L,C^{(n-1)}(j)^n_{[n-1]}).
\]
Writing $L\ten \Delta^j:= L\ten_{\Z}C_{\bt}(\Delta^j,\Z)$ as the tensor product with simplicial chains, we can 
 represent $o_e(u)$ as a morphism $\{L \ten\Delta^j \to C^{(n-1)}(j)^n_{[n-1]}\}_{j \in \Delta}$ of cosimplicial $A$-modules in chain complexes, since $L$ is cofibrant, and hence $j\mapsto L \ten \Delta^j$ is a cofibrant resolution in the Reedy model category $dg_+\Alg(R)^{\Delta}$. 
We then rewrite this as a morphism $ \{ \sigma_{\ge 0}(L_{[1-n]} \ten \Delta^j) \to C^{(n-1)}(j)^n\}_{j \in \Delta}   $, where $\sigma$ denotes brutal truncation.

We then set $C^{(n)}(j)$ to be the pushout of the diagram
\[
  C^{(n-1)}(j) \la \Omega^{\bt}_{A\<\sigma_{\ge 0}(L_{[1-n]}\ten \Delta^j )^{[-n]}\> } \to  \Omega^{\bt}_{A\<\cone\sigma_{\ge 0}( (L_{[1-n]}\ten \Delta^j ))^{[-n]}\>}, 
\]
the cone being taken in the chain direction, so that $\Hom_{DG^+dg_+\Alg(R)}( C^{(n)}(j),B)$ is the fibre product of the diagram
\[
\xymatrix@R=2ex{
    \Hom_{DG^+dg_+\Alg(R)}( C^{(n-1)}(j),B)\ar_{o_e(u)}[d] & \Hom_{dg_+\Mod(A)}(\cone\sigma_{\ge 0}(L_{[1-n]}\ten \Delta^j ),B^n)\ar[ld]\\
    \Hom_{dg_+\Mod(A)}( \sigma_{\ge 0}L_{[1-n]} \ten \Delta^j,  B^n ) . 
}
    \]
Since the second map in the fibre product is surjective and cones are acyclic, by substituting a simplicial fibrant resolution of $B$ to calculate function complexes as in \cite[\S 5.4]{Hovey},  this gives us a homotopy fibre sequence 
\[
 \xymatrix@R=2ex{\ho\LLim_{j \in \Delta} \map_{\Omega_A^{\bt}}(C^{(n)}(j),B)\ar[r] & \ho\LLim_{j \in \Delta} \map_{\Omega_A^{\bt}}(C^{(n-1)}(j),B)\ar[d]\\ &  \map_{dg\Mod(A)}(L, B^n_{[n-1]}),
 }
\]
where we write $\map_{\Omega_A^{\bt}}:=\map_{\Omega_A^{\bt}\da DG^+dg_+\Alg(R)} $.
Hence by induction, for the given morphism $f \co \Omega^{\bt}_A \to B$, we have
\[
 \ho\LLim_{j \in \Delta} \map_{\Omega_A^{\bt}\da DG^+dg_+\Alg(R)}(C^{(n)}(j),B)\simeq \hat{F}_x^{(n)}(B) \by^h_{\map_{dg_+\Alg(R)}(A,B^0)}\{f^0\},  
\]
and thus
\[
\ho\LLim_{j \in \Delta} \map_{DG^+dg_+\Alg(R)}(C^{(n)}(j),B)\simeq \hat{F}_x^{(n)}(B). 
\]

It only remains to show that the maps $C(j) \to C(k)$ are all homotopy  \'etale. We begin by observing that $\Omega^1_{C^{(0)}(j)}\ten_{C^{(0)}(j)^{\oL,e}}A^{\oL,e}$ is the double complex $\Omega^1_A \oplus  (\Omega^1_A)^{[-1]}$ for all $j$, with cochain differential given by the identity. Then we see that by construction $\Omega^1_{C^{(n)}(j)}\ten_{C^{(n)}(j)^{\oL,e}}A^{\oL,e}$ is an extension of $\Omega^1_{C^{(n-1)}(j)}\ten_{C^{(n-1)}(j)^{\oL,e}}A^{\oL,e}$ by the cochain chain complex $\sigma_{\ge 1-n}(L \ten \Delta^j)_{[-n]}^{[-n]} \xra{\id} \sigma_{\ge 1-n}(L \ten \Delta^j)_{[-n]}^{[-1-n]}$. %%_{[-n]} rather than $_{[1-n]}$ because cone.

By looking at how the constructions above behave for $B$ of the form $A \oplus M\eps$, we see that the extensions in each cochain degree are given by the obvious surjective maps $\sigma_{\ge 1-n}(L \ten \Delta^j) \to \sigma_{\ge 2-n}(L \ten \Delta^j)$, so it follows that we have  quasi-isomorphisms
\[
 (\Omega^1_{C(j)}\ten_{C(j)^{\oL,e}}A^{\oL,e} )^n \simeq \begin{cases}
                                \Omega^1_A & n=0\\
                                \cone(\sigma_{\ge 0}(L\ten \Delta^j) \to  \Omega^1_A) & n=1\\
                                ((L \ten \Delta^j)_{1-n})_{[-1]} & n \ge 2,   
                               \end{cases}
\]
of chain complexes,
with $\pd \co \H_1( (\Omega^1_{C(j)})^n) \to \H_1( (\Omega^1_{C(j)})^{n+1})$ being given by the structural differential $\delta \co (L \ten \Delta^j)_{1-n} \to (L \ten \Delta^j)_{-n}$. Since product total complexes respect quasi-isomorphisms of bounded above chain complexes, we then have 
\[
 \Tot^{\Pi}\sigma^{\le n}(\Omega^1_{C(j)}\ten_{C(j)^{\oL,e}}A^{\oL,e} ) \simeq \sigma_{\ge 1-n}(L \ten \Delta^j) 
\]
for $n \ge 1$, and hence $ \Tot^{\Pi}\sigma^{\le n}(\Omega^1_{C(j)}\ten_{C(j)^{\oL,e}}A^{\oL,e} ) \simeq L$ for all $n > M$. This gives the quasi-isomorphism $\Tot^{\Pi}\sigma^{\le n}(\Omega^1_{C(j)}\ten_{C(j)^{\oL,e}}A^{\oL,e} ) \simeq \Tot^{\Pi}\sigma^{\le n}(\Omega^1_{C(j)}\ten_{C(k)^{\oL,e}}A^{\oL,e} )$ required, since the inverse systems have stabilised.
\end{proof}

\begin{remark}
If $F$ is represented by a Reedy cofibrant cosimplicial diagram $O(X) \in dg_+\Alg(R)^{\Delta}$, as in \S \ref{derivedgeomprestacksn}, then $D_*F$ is represented by the cosimplicial diagram  $j \mapsto D^*O(X^{\Delta^j})$, so another representing cosimplicial diagram for the functor $\hat{F}_x$ of Proposition \ref{cosimplicialresnprop} is given by $j \mapsto D^*O(X^{\Delta^j})\coprod^{\oL}_{\Omega^{\bt}_{O(X_j)}}\Omega^{\bt}_A=:C'(j)$. 

This looks quite different from the diagram $C(\bt)$ constructed in the proof of  Proposition \ref{cosimplicialresnprop}: whereas
 $U':=\Omega^1_{C'(j)}\ten_{C'(j)^{\oL,e}}\H_0A$ has $\H_n(U')^m=0$ for $m >0$ and $n \ne 0$, the double complex 
$U=\Omega^1_{C(j)}\ten_{C(j)^{\oL,e}}\H_0A$ has $\H_n(U)^m=0$ for $m>1$ and $n\ne 1$. The extra $m=1$ terms enable the construction in Proposition \ref{cosimplicialresnprop} to avoid trying to  embed $\Spec A$ in a derived affine scheme submersive over $F$.
\end{remark}

% 
% \begin{definition}\label{LintFdef}
% Define $\oL dg\CAlg(R)_{c, \onto}$, $\oL dg\CAlg(R)$, $\oL\int F$, and $\oL s\Set$   to be the $\infty$-categories obtained by localising the respective %$\infty$-
% categories at quasi-isomorphisms or weak equivalences.
% \end{definition}

%remember that this is basically $(U/F)_{\dR}$.

\begin{definition}\label{Frigdef}
 Given a stacky DGAA $B \in DG^+dg_+\Alg(R)$  for which the chain complexes
$
( \oL\Omega^1_B\ten^{\oL}_{B^{\oL,e}}(B^0)^{\oL,e})^i
$
are acyclic for all $i > q$, and  a  homogeneous functor $F\co dg_+\Alg(R) \to s\Set$   with a cotangent complex $L_F(B^0,x)$ at a point $x \in F(B^0)$, we say that a point $y \in D_*F(B)$ lifting $x \in F(B^0)$ is \emph{rigid} if the induced morphism
 \[
  L_F(B^0,x)\to \Tot \sigma^{\le q} \oL\Omega^1_B\ten^{\oL}_{B^{\oL,e}}(B^0)^{\oL,e}
 \]
is a quasi-isomorphism of $B^0$-bimodules. 

We denote by $(D_*F)_{\rig}(B) \subset D_*F(B)$ the space of rigid points (a union of path components).
\end{definition}

\begin{remark}
 Note that for  a point $y \in D_*F(B)$ to be  rigid is the same as saying that it does not deform: for any nilpotent surjection $e \co C \to B$ with a point $z \in D_*F(C)$ lifting $y$, the map $e$  has an essentially unique section $s$ with $s(y) \simeq z$. 
\end{remark}

\begin{definition}
Define the $\infty$-category $\oL dg^+DG^+\Aff(R)$ by localising $DG^+dg_+\Alg(R)^{\op}$ at levelwise weak equivalences, and let $\oL dg^+DG^+\Aff(R)^{\et}$ be the full $2$-sub-$\infty$-category of $ \oL dg^+DG^+\Aff(R) $ with the same objects but only spaces of homotopy  \'etale morphisms (so $\map_{ \oL dg^+DG^+\Aff(R)^{\et}}(A,B) \subset \map_{ \oL dg^+DG^+\Aff(R)}(A,B)$ is a union of path components). We then denote the $\infty$-category of simplicial presheaves on $\oL dg^+DG^+\Aff(R)$  by $\oL (dg^+DG^+\Aff(R)^{\et})^{\wedge}$.
  \end{definition}

Now observe that $(D_*F)_{\rig}$ defines a functor on  $\oL dg^+DG^+\Aff(R)^{\et}$, since rigidity is preserved by homotopy  \'etale morphisms.

\begin{corollary}\label{etsitecor}
  If a homotopy-preserving homogeneous functor $F\co dg_+\Alg(R) \to s\Set$   has bounded below cotangent complexes $L_F(A,x)$ at all points $x \in F(A)$ for all $A$, then
for any homotopy-preserving functor $G \co DG^+dg_+\Alg(R) \to s\Set$, we have a natural weak equivalence
\begin{align*}
 \map_{\oL dg^+DG^+\Aff(R)^{\wedge}}(D_*F,G) &\simeq \map_{\oL (dg^+DG^+\Aff(R)^{\et})^{\wedge}}((D_*F)_{\rig}, \theta_*G)\\
\int_{B \in \cS t}^h \map_{s\Set}(D_*F(B),G(B))  &\simeq \int_{B \in \cS t^{\et}}^h \map_{s\Set}((D_*F)_{\rig}(B),G(B)),
  \end{align*}
  for $\cS t:=\oL DG^+dg_+\Alg(R)$ and  $\theta \co  DG^+dg_+\Alg(R)^{\et} \to DG^+dg_+\Alg(R)$ the inclusion functor.
   \end{corollary}
\begin{proof}
The resolution of Proposition \ref{cosimplicialresnprop} lies in $dg^+DG^+\Aff(R)^{\et}$,   and the canonical elements $u_j \in D_*F(C(j))$ are all rigid. That proposition gives us an expression $ \hat{F}_x \simeq \ho\LLim_{j \in \Delta} \oR \Spec C(j)$, and similarly setting 
\[
\hat{F}_{x,\rig}(B):= (D_*F)_{\rig}(B)\by^h_{F(B^0),x^*}\map(A,B^0)  
\]
gives $
 \hat{F}_{x,\rig} \simeq \ho\LLim_{j \in \Delta} (\oR \Spec C(j))_{\rig}$,
 for  $(\oR \Spec C(j))_{\rig} \in \oL (dg^+DG^+\Aff(R)^{\et})^{\wedge}$ the prestack represented by $C(j)$.

Now $\theta^*\co \oL dg^+DG^+\Aff(R)^{\wedge} \to \oL (dg^+DG^+\Aff(R)^{\et})^{\wedge}$ has a derived left adjoint $\oL \theta^*$ which sends $(\oR \Spec C)_{\rig}$ to $\oR \Spec C$, so we deduce that $\oL \theta^*  \hat{F}_{x,\rig} \simeq \hat{F}_x  $.

 Since homotopy colimits in $s\Set$ are universal (immediate for coproducts, while homtopy pushouts follow from \cite{puppeRmkHtpy}), 
%   \cite{rezkHomotopyToposSketch},  homotopy colimits in $s\Set$ are universal. ETS for arb coprods and pushouts. Former obvious, latter in
% [Pup74]   Volker Puppe,A remark on “homotopy fibrations”, Manuscripta Math.12(1974), 113–120.
we have that 
\begin{align*}
 \ho\LLim_{(x,A)} \hat{F}_x(B) &\simeq D_*F(B)\by_{F(B^0)}^h\ho\LLim_{(x,A)}\map_{dg_+\Alg(R)}(A,B^0)\\
&\simeq D_*F(B)\by_{F(B^0)}^hF(B^0)\\
&\simeq D_*F(B),
 \end{align*}
 where $(x,A)$ runs over objects of $dg^+DG^+\Aff(R)\da F$,
and similarly 
\[
\ho\LLim_{(x,A)} \hat{F}_{x,\rig}(B)\simeq (D_*F)_{\rig}(B).
\]
Since derived left adjoints commute with homotopy colimits, this gives
\[
 D_*F \simeq \oL\theta^*(D_*F)_{\rig}, 
\]
which is equivalent to the desired statement.
\end{proof}
 
 %%this is as good as CPTVV manage for symplectic structures:
\begin{proposition}\label{replaceprop}
 Given a  derived $\infty$-geometric    Artin  NC prestack $F$ and any homotopy-preserving functor $G \co dg_+\Alg(R) \to s\Set$, we have a natural weak equivalence
 \[
  \map_{\oL DG^+\Aff(R)^{\wedge}}(F,G) \to \map_{\oL dg^+DG^+\Aff(R)^{\wedge}}(D_*F,D_*G).
 \]
\end{proposition}
\begin{proof}
If $F$ is sqc, then it is represented by a hypergroupoid $X_{\bt}$, and then $D_*F$ is represented by the cosimplicial stacky DGAA $j \mapsto D^*O(X^{\Delta^j})$, so we have
\begin{align*}
  \map_{\oL dg^+DG^+\Aff(R)^{\wedge}}(D_*F,D_*G) &\simeq \ho \Lim_{j \in \Delta} (D_*G)(D^*O(X^{\Delta^j}))\\
  &\simeq \ho \hskip -2.7ex \Lim_{(i,j)  \in \Delta\by \Delta} G(D^iD^*O(X^{\Delta^j}))\\
  &\simeq \map_{\oL DG^+\Aff(R)^{\wedge}}( \ho \hskip -2.7ex\LLim_{(i,j)  \in \Delta\by \Delta} \oR\Spec^{nc} D^iD^*O(X^{\Delta^j}),G),
\end{align*}
and the argument of \cite[Proposition \ref{poisson-replaceprop}]{poisson} gives a weak equivalence $\oR\Spec^{nc} DD^*O(X^{\Delta^{\bt}}) \simeq X$ of simplicial presheaves, from which the equivalence follows.

In general, we write $F$ as a filtered colimit $\LLim_{\alpha}F_{\alpha}$ of sqc  derived $n$-geometric    Artin  NC prestacks, and the key observation to make is that if $B$ is bounded in the cochain direction, then the homotopy limit defining $D_*F(B)$ is effectively finite, so $D_*F(B)\simeq \LLim_{\alpha}D_*F_{\alpha}(B)$. We then have
\begin{align*}
 \map(D_*F,D_*G)&\simeq \int_{B \in \oL DG^+dg_+\Alg(R)}^h \map_{s\Set}(D_*F(B),G(B))\\
 &\simeq \ho\Lim_n \int_{B \in \oL DG^+dg_+\Alg(R)}^h \map_{s\Set}(D_*F(B),G(B^{\le n}))\\
  &\simeq \ho\Lim_n \int_{B \in \oL DG^+dg_+\Alg(R)}^h \map_{s\Set}(D_*F(B^{\le n}),G(B^{\le n}))\\
  &\simeq\ho\Lim_n\ho\Lim_{\alpha}\int_{B \in \oL DG^+dg_+\Alg(R)}^h \map_{s\Set}(D_*F_{\alpha}(B^{\le n}),G(B^{\le n}))\\
  &\simeq\ho\Lim_n\ho\Lim_{\alpha} \map(F_{\alpha},G)\\
   & \simeq\map(F,G).\qedhere
\end{align*}
\end{proof}
 
The significance of  Corollary \ref{etsitecor} and Proposition \ref{replaceprop} arises when $G$ classifies some structure (for example the shifted pre-bisymplectic structures of \cite{NCpoisson}), with maps  from $F$ to $G$ then corresponding to the space of such structures on  $F$.  Proposition \ref{replaceprop} allows us to reinterpret this as a structure on $D_*F$, with Corollary \ref{etsitecor} allowing us to reduce further to the rigid functor $(D_*F)_{\rig}$. Consequently, the full functoriality of $G$ is not needed to formulate structures on $F$, since we only need homotopy \'etale functoriality of $D_*G$ on stacky DGAAs. This then permits comparison with structures (such as the shifted double Poisson structures of \cite{NCpoisson}) which are only functorial with respect to homotopy \'etale morphisms of stacky DGAAs.
 
\bibliographystyle{alphanum}
\bibliography{references.bib}

\end{document}

%% file: newcommands.tex
\newcommand\da{\!\downarrow\!}

\newcommand\la{\leftarrow}

\newcommand\id{\mathrm{id}}

\newcommand\ten{\otimes}

\newcommand\eps{\epsilon}

\renewcommand\H{\mathrm{H}}
\newcommand\z{\mathrm{Z}}

\newcommand\N{\mathbb{N}}
\newcommand\Z{\mathbb{Z}}
\newcommand\Q{\mathbb{Q}}

\newcommand\R{\mathbb{R}}

\newcommand\bG{\mathbb{G}}

\newcommand\bL{\mathbb{L}}

\newcommand\C{\mathcal{C}}

\newcommand\cA{\mathcal{A}}
\newcommand\cB{\mathcal{B}}

\newcommand\cD{\mathcal{D}}

\newcommand\cI{\mathcal{I}}

\newcommand\cK{\mathcal{K}}

\newcommand\cP{\mathcal{P}}

\newcommand\cS{\mathcal{S}}

\newcommand\cW{\mathcal{W}}

\newcommand\sE{\mathscr{E}}

\newcommand\sO{\mathscr{O}}
\newcommand\sP{\mathscr{P}}

\renewcommand\L{\Lambda}

\newcommand\g{\mathfrak{g}}

\newcommand\cHom{\mathcal{H}\!\mathit{om}}

\newcommand\Ho{\mathrm{Ho}}

\newcommand\Alg{\mathrm{Alg}}

\newcommand\comm{\mathrm{com}}

\newcommand\Mod{\mathrm{Mod}}

\newcommand\Hom{\mathrm{Hom}}
\newcommand\Map{\mathrm{Map}}
\newcommand\map{\mathrm{map}}

\newcommand\HHom{\underline{\mathrm{Hom}}}
\newcommand\EEnd{\underline{\mathrm{End}}}

\newcommand\Ext{\mathrm{Ext}}
\newcommand\EExt{\mathbb{E}\mathrm{xt}}
\newcommand\End{\mathrm{End}}

\newcommand\Der{\mathrm{Der}}

\newcommand\Iso{\mathrm{Iso}}

\newcommand\cone{\mathrm{cone}}
\newcommand\cocone{\mathrm{cocone}}
\newcommand\dg{\mathrm{dg}}
\newcommand\per{\mathrm{per}}

\newcommand\Ob{\mathrm{Ob}}

\newcommand\Spec{\mathrm{Spec}\,}

\newcommand\Set{\mathrm{Set}}

\newcommand\Cat{\mathrm{Cat}}

\newcommand\Mor{\mathrm{Mor}\,}

\newcommand\Aff{\mathrm{Aff}}

\newcommand\<{\langle}
\renewcommand\>{\rangle}
\newcommand\Lim{\varprojlim}
\newcommand\LLim{\varinjlim}

\newcommand\ho{\mathrm{ho}\!}
\newcommand\into{\hookrightarrow}
\newcommand\onto{\twoheadrightarrow}

\newcommand\xra{\xrightarrow}
\newcommand\xla{\xleftarrow}

\newcommand\pr{\mathrm{pr}}

\newcommand\bt{\bullet}
\newcommand\by{\times}

\newcommand\ddef{\mathrm{Def}}

\newcommand\Perf{\mathrm{Perf}}

\newcommand\GL{\mathrm{GL}}

\newcommand\Mat{\mathrm{Mat}}
\newcommand\et{\acute{\mathrm{e}}\mathrm{t}}

\newcommand\Tot{\mathrm{Tot}\,}

\newcommand\pd{\partial}

\newcommand\rig{\mathrm{rig}}

\newcommand\op{\mathrm{opp}}

\newcommand\co{\colon\thinspace}

\newcommand\oR{\mathbf{R}}

\newcommand\oL{\mathbf{L}}

\newcommand\uleft\underleftarrow
\newcommand\uline\underline
\newcommand\uright\underrightarrow

%% file: newthms.tex
\newtheorem{theorem}{Theorem}[section]
\newtheorem{proposition}[theorem]{Proposition}
\newtheorem{corollary}[theorem]{Corollary}

\newtheorem{lemma}[theorem]{Lemma}
\newtheorem*{theorem*}{Theorem}
\newtheorem*{proposition*}{Proposition}
\newtheorem*{corollary*}{Corollary}
\newtheorem*{lemma*}{Lemma}
\newtheorem*{conjecture*}{Conjecture}

\theoremstyle{definition}
\newtheorem{definition}[theorem]{Definition}
\newtheorem*{definition*}{Definition}

\newtheorem*{notation*}{Notation}

\theoremstyle{remark}
\newtheorem{example}[theorem]{Example}

\newtheorem{remark}[theorem]{Remark}
\newtheorem{remarks}[theorem]{Remarks}

\newtheorem*{example*}{Example}
\newtheorem*{examples*}{Examples}
\newtheorem*{remark*}{Remark}
\newtheorem*{remarks*}{Remarks}
\newtheorem*{exercise*}{Exercise}
\newtheorem*{property*}{Property}
\newtheorem*{properties*}{Properties}